\documentclass[preprint]{imsart}

\usepackage{amsthm,amsmath, amssymb}

\RequirePackage[numbers]{natbib}

\allowdisplaybreaks

\usepackage{graphicx}

\usepackage[a4paper, margin=1in]{geometry}

\startlocaldefs
\numberwithin{equation}{section}
\theoremstyle{plain}
\newtheorem{thm}{Theorem}[section]
\newtheorem{defi}[thm]{Definition}
\newtheorem{exmp}[thm]{Example}

\newtheorem{prop}[thm]{Proposition}
\newtheorem{rem}[thm]{Remark}
\newtheorem{cor}[thm]{Corollary}

\newcommand{\R}{\mathbb{R}}
\newcommand{\N}{\mathbb{N}}
\newcommand{\F}{\mathcal{F}}

\newcommand{\h}{\mathfrak{h}}
\newcommand{\f}{\mathfrak{f}}
\newcommand{\g}{\mathfrak{g}}

\newcommand{\inth}{\underline{h}}
\newcommand{\intg}{\underline{g}}

\newcommand{\Osc}{\textnormal{Osc}}

\endlocaldefs

\begin{document}

\begin{frontmatter}

% "Title of the paper"
\title{It\^{o}'s Formula for  Gaussian Processes with Stochastic Discontinuities}
\runtitle{It\^{o}'s Formula for  Gaussian Processes}

% indicate corresponding author with \corref{}
% \author{\fnms{John} \snm{Smith}\corref{}\ead[label=e1]{smith@foo.com}\thanksref{t1}}
% \thankstext{t1}{Thanks to somebody} 
% \address{line 1\\ line 2\\ printead{e1}}
% \affiliation{Some University}

\author{\fnms{Christian} \snm{Bender}\corref{} \ead[label=e1]{bender@math.uni-sb.de}}
\thankstext{t1}{This is a preprint version of an article 
published in the Annals of Probability, 2020, Vol. 48, 458--492.
The author thanks the associate editor and the referee for their valuable comments, which led to significant improvements of the paper.} 
\address{Department of Mathematics\\ Saarland University\\ PO Box 151150\\ 66041 Saarbr\"ucken\\Germany\\ \printead{e1}}
\affiliation{Saarland University }

\runauthor{C. Bender }

\begin{abstract}
We introduce a Skorokhod type integral and prove an It\^o formula for a wide class of Gaussian processes which may exhibit stochastic discontinuities. Our It\^o formula unifies and  extends
the classical one for general (i.e., possibly discontinuous) Gaussian martingales in the sense of It\^o integration and the one for stochastically continuous Gaussian non-martingales in the Skorokhod sense, which was first 
derived in Al\`os et al. (Ann.  Probab. 29, 2001). 
\end{abstract}



\end{frontmatter}

\section{Introduction}

Since the pioneering work of \citet{AMN}, It\^o's formula for Gaussian processes in the sense of Skorokhod type integration has been developed in a series of papers. The generic formula reads as follows: If $X$ is a
centered Gaussian 
process with  variance function $V$, which is assumed to be of bounded variation and continuous, and $F$ is sufficiently smooth, then
\begin{equation}\label{eq:Ito_cts}
 F(X_T)=F(X_0)+\int_0^T F'(X_s) d^\diamond X_s + \frac{1}{2} \int_0^T F''(X_s) dV(s).
\end{equation}
This formula has been shown to be valid under structural assumptions on a kernel representation of $X$ with respect to a Brownian motion \citep{AMN, MV, Leb},  on the covariance 
function \citep{KRV, LN, HJT, AK},  and on the quadratic variation of $X$ \citep{NT, NT2}. 

An important contribution by \citet{MV} clarifies that  the It\^o formula \eqref{eq:Ito_cts} can even hold for certain classes of Gaussian processes, whose paths feature discontinuities of the second type. The Gaussian processes 
in \cite{MV} and in all the other references mentioned above are, however, stochastically continuous. In view of Theorem 1.1 in \citet{Sam} this property rules out the possibilty that the paths of $X$ have discontinuities of the first kind (i.e., jumps) only. 
In this paper, we derive, for the first time, the general form of the Gaussian It\^o formula in the presence of stochastic discontinuities. Assuming, for the sake of exposition, that $X$ is stochastically right-continuous with left limits, 
our It\^o formula takes the following form:
\begin{eqnarray}\label{eq:Ito_rcll}
 F(X_T)&=&F(X_0)+\int_{0+}^T F'(X_{s-}) d^\diamond X_s + \frac{1}{2} \int_{0+}^T F''(X_{s-}) dV^c(s) \nonumber \\
&& + \sum_{s\in D_X \cap (0,T]} \Bigl( F(X_s)- F(X_{s-})-F'(X_{s-}) (X_s-X_{s-})  \\ &&  \quad\quad\quad\quad\quad\quad  + F''(X_{s-})E[(X_s-X_{s-})X_{s-}]\Bigr). \nonumber 
\end{eqnarray}
Here, $V^c$ denotes the continuous part of the variance function and the sum runs over the set $D_X$ of deterministic time points in $(0,T]$, at which $X$ exhibits a stochastic discontinuity. This set of stochastic discontinuities can be determined by the covariance function of $X$.

 We wish to emphasize the following features:
\begin{itemize}
\item[i)] Our key assumption is a regularity assumption of the elements in the Cameron-Martin space of $X$ in terms of their quadratic variation. The crucial role, which regularity in the Cameron-Martin space plays for deriving the Gaussian It\^o
formula, has apparently been largely unnoticed in the literature. On the one hand, it opens the door to extend the Gaussian It\^o formula to processes with jumps. On the other hand, it can easily be checked in all references (which we are aware of),
in which the It\^o formula \eqref{eq:Ito_cts} is proved for stochastically continuous processes. Hence, our approach has a unifying and generalizing character at the same time.
\item[ii)] The second derivative term in the jump component of our It\^o formula depends on the covariance structure of the Gaussian process $X$ via the expression $E[X_s X_{s-}]$. This is in contrast to the stochastically continuous case, in which only 
the variance function appears in the corresponding formula \eqref{eq:Ito_cts}. If $X$ is a Gaussian martingale with jumps, then this extra second derivative term clearly vanishes, and the It\^o formula \eqref{eq:Ito_rcll} reduces to its classiscal form:
\begin{eqnarray}\label{eq:Ito_mart}
 F(X_T)&=&F(X_0)+\int_{0+}^T F'(X_{s-}) dX_s + \frac{1}{2} \int_{0+}^T F''(X_{s-}) d [X]^c_s  \nonumber \\ &&+ \sum_{s\in  (0,T]} F(X_s)- F(X_{s-})-F'(X_{s-}) (X_s-X_{s-}),
\end{eqnarray}
(where we exploit that the continuous part $V^c$ of the variance coincides with the continuous part of the quadratic variation $[X]$ in the Gaussian case and
that, by \cite{JM},  all the jumps of a Gaussian martingale occur at the deterministic times of stochastic discontinuities).
\item[(iii)] The Skorokhod type integration, which we develop in this paper, and It\^o's formula can be extended to conditionally Gaussian processes. We explain this construction in detail for fractional Brownian motion subordinated 
by an increasing L\'evy process. This class of processes shares many important properties of fractional Brownian motion like stationary increments and short/long memory (in dependence of the Hurst parameter and the L\'evy subordinator). Additionally, 
subordinated fractional Brownian motion features jumps and heavier tails than Gaussian, which makes it a promising model for real-world phenomena, see e.g. \citet{Gal,Mal,DHJ}. To the best of our knowledge, 
 this paper constitutes the first approach 
to a Skorokhod type stochastic calculus for subordinated fractional Brownian motion (or, more generally, conditionally Gaussian processes with jumps).   
\end{itemize}

The paper is organized as follows.  In Section \ref{sec_regulated}, we first recall some preliminaries on Gaussian processes and 
then introduce  the concept of a weakly 
regulated process. It allows giving a meaning to the one-sided limits $X_{s\pm}$ without imposing any path regularity assumption on $X$, provided that the 
functions in the Cameron-Martin space 
of $X$ are regulated.  These one-sided limits appear in our general version of It\^o's formula in Theorem \ref{thm:Ito}.
 In Section \ref{sec_regulated}, we  also study the set of stochastic discontinuities for weakly regulated Gaussian processes. In Section \ref{sec_Malliavin}, we recall the definition of the Malliavin derivative 
and of the divergence operator, with an emphasis on the role of the $S$-transform and of regularity in the Cameron-Martin space. The observations in Section \ref{sec_Malliavin} serve 
as a main motivation for our notion of Wick-Skorokhod integration, which we 
define in Section \ref{sec_integral}  in terms of the $S$-transform and the Henstock-Kurzweil integral. This $S$-transform approach has already been 
adopted in \cite{BeBernoulli, Be2014}, and actually, can be viewed as a main tool 
for studying Hitsuda-Skorokhod integration in a white noise framework as in \cite{Leb} and the references therein. We also show, that our integrals extends the classical stochastic It\^o integral 
for predictable integrands to anticipating integrands in the case that $X$ is a Gaussian martingale. 
   After these preparations we can state 
and prove our It\^o formula in its general form in Section \ref{sec_ito}. Section \ref{sec:fBm} is devoted to subordinated fractional Brownian motion. We finally explain how to check the required structural assumption on the Cameron-Martin space in Section \ref{sec_reg}
 and compare 
our assumptions to the ones imposed in the existing literature in Section \ref{sec_comparison}.  In the appendix, we provide,
following ideas of \cite{No}, a chain rule for the Henstock-Kurzweil 
integral, which is required in our proof of It\^o's formula.

\section{Weakly regulated processes}\label{sec_regulated}

The main purpose of this section is to give a meaning to the one-sided limits $X_{s\pm}$, which occur in the It\^o formula (Theorem \ref{thm:Ito}), without imposing path regularity assumptions on $X$.
Before doing so, let us recall some facts on Gaussian processes for ready reference.

Suppose $(X_t)_{0\leq t \leq T}$ is a centered Gaussian process on a  complete probability space $(\Omega,\F,P)$ with variance function $V(t):=E[X(t)^2]$. We denote by $H_X$ the first chaos of $X$, i.e. the Gaussian 
Hilbert space, which is obtained by taking the closure of the linear span of
$
\{X_t;\; t\in[0,T]\}
$
in $L^2(\Omega,\F,P)$. To each element $h\in H_X$, one can associate a function
$$
\inth: [0,T] \rightarrow \R,\quad t \mapsto E[X_th].
$$
The space of functions
$$
CM_X:=\{\inth;\; h\in H_X\}
$$
is called the \emph{Cameron-Martin space} associated to $X$. As, by definition, the set $
\{X_t;\; t\in[0,T]\}
$ is total in $H_X$, the map 
$$
H_X\rightarrow CM_X,\quad h\mapsto \inth
$$
is bijective. It becomes an isometry, if one equips the Cameron-Martin space with the inner product
$$
\langle \inth,\intg\rangle_{CM_X}:=E[h g].
$$
The \emph{Wick exponential} of $h\in H_X$ is defined to be 
$$
\exp^\diamond(h):=\exp\{h-E[h^2]/2\}.
$$ 
If $\mathcal{A}$ is a dense subset of $H_X$, then,  by Corollary 3.40 in \cite{Ja}, the set 
$$
\{\exp^\diamond(h);\; h\in \mathcal{A} \} 
$$
is total in $(L^2_X):=L^2(\Omega,\F^X,P)$, where $\F^X$ is the completion by the $P$-null sets of the $\sigma$-field generated by $X$. 
Hence, every random variable $\xi \in (L^2_X)$ is uniquely determined by its \emph{$S$-transform} restricted to $\mathcal{A}$,
$$
(S\xi)(h):=E[\exp^\diamond(h) \xi],\quad h\in \mathcal{A}.
$$
This means, the identity $\xi=\eta$ is valid in $(L^2_X)$, if and only if   $(S\xi)(h)=(S\eta)(h)$ for every $h\in \mathcal{A}$.
We also recall the following straightforward identities ($g,h\in H_X,\; t\in [0,T]$)
\begin{eqnarray}\label{eq:Stransform1}
 \ \quad \exp^\diamond(h)\exp^\diamond(g)&=&e^{E[gh]}\exp^\diamond(g+h),\quad (S \exp^\diamond(g))(h)=e^{E[gh]},\\ 
 (Sg)(h)&=&E[gh],\quad (SX_t)(h)=\inth(t). \nonumber 
\end{eqnarray}
More, generally, we note that, by a classical result on Gaussian change of measure,
\begin{equation}\label{eq:shift}
 (S G(g_1,\ldots,g_D))(h)= E[G(g_1 +E[g_1h],\ldots, g_D +E[g_Dh])]
\end{equation}
for every $g_1,\ldots,g_D,h\in H_X$ and measurable $G:\R^D\rightarrow \R$, provided that the expectation on the right-hand side exists, see e.g. \citet[Theorem 14.1]{Ja}.

After these preliminaries, we now define the notion of a weakly regulated process.
\begin{defi}
 A stochastic process $Y:[0,T]\rightarrow (L^2_X)$ is called \emph{weakly regulated}, if for every $s\in (0,T]$ there is a random variable 
$Y_{s-}\in (L^2_X)$ such that for every sequence $(s_n)$ which converges to $s$ from the left 
$$
\lim_{n\rightarrow \infty} Y_{s_n}=Y_{s-} \textnormal{ weakly in } (L^2_X),
$$
and, if 
for every $s\in [0,T)$ there is a random variable 
$Y_{s+}\in (L^2_X)$ such that for every sequence $(s_n)$ which converges to $s$ from the right 
$$
\lim_{n\rightarrow \infty} Y_{s_n}=Y_{s+} \textnormal{ weakly in } (L^2_X),
$$
\end{defi}
For a weakly regulated process, we shall apply the convention $Y_{0-}:=Y_0$ and $Y_{T+}:=Y_T$. Moreover, we write
$$
\Delta Y_s=Y_{s+}-Y_{s-},\quad \Delta^+ Y_s=Y_{s+}-Y_{s},\quad \Delta^-Y_s=Y_s-Y_{s-}.
$$
The analogous notation is used for the jumps of deterministic regulated functions.

Weakly regulated processes can be characterized via the $S$-transform as follows:
\begin{prop}\label{prop:weakly_regulated}
 Suppose $Y:[0,T]\rightarrow (L^2_X)$. If $Y$ is weakly regulated, then the map $t\mapsto E[Y_t^2]$ is bounded and, for every $h\in H_X$, the map 
$t\mapsto (SY_t)(h)$ is regulated (i.e., has limits from the left and from the right).

Conversely, assume that $\mathcal{A}$ is a dense subset of $H_X$. If the map $t\mapsto E[Y_t^2]$ is bounded and, for every $h\in \mathcal{A}$, the map 
$t\mapsto (SY_t)(h)$ is regulated, then $Y$ is weakly regulated.
\end{prop}
Before we prove this proposition, let us state the following corollary concerning the Gaussian process $X$.
\begin{cor}\label{cor:weakly_regulated}
 The Gaussian process $(X_t)_{t\in [0,T]}$ is weakly regulated, if and only if its variance function $V$ is bounded and the Cameron-Martin space $CM_X$ has a dense subset consisting 
of regulated functions. In this case the weak $(L^2_X)$-limits $X_{s\pm}$ belong to the first chaos $H_X$.
\end{cor}
\begin{proof}
 In view of Proposition \ref{prop:weakly_regulated}, the first statement is an immediate consequence of \eqref{eq:Stransform1}. For the second one, let $s\in[0,T)$ and denote by 
$\pi_{H_X}$ the orthogonal projection from $(L^2_X)$ on $H_X$. Then,
\begin{eqnarray*}
 E[|X_{s+}-\pi_{H_X}(X_{s+})|^2]&=& E[X_{s+}(X_{s+}-\pi_{H_X}(X_{s+}))]\\ &=&\lim_{n\rightarrow \infty}  E[X_{s+\frac{1}{n}}(X_{s+}-\pi_{H_X}(X_{s+}))].
\end{eqnarray*}
The right-hand side equals 0, because $X_{s+1/n}\in H_X$. Hence, $X_{s+}\in H_X$, and the left-sided limits can be handled analogously.
\end{proof}

\begin{proof}[Proof of Proposition \ref{prop:weakly_regulated}]
 Suppose first that $Y$ is weakly regulated. Fix $h\in H_X$ and $s\in (0,T]$. Then, for every  sequence $(s_n)$ which converges to $s$ from the left,
$$
\lim_{n\rightarrow \infty} (SY_{s_n})(h)=\lim_{n\rightarrow \infty} E[Y_{s_n} \exp^\diamond(h)] = E[Y_{s-} \exp^\diamond(h)]
$$
Hence, $t\mapsto (SY_t)(h)$ has a limit from the left at $s$. In the same way, one observes that it has a limit from the right at every $s\in [0,T)$. We next prove the boundedness of $t\mapsto E[Y_t^2]$
by contradiction. Thus, suppose to the contrary that this function is unbounded, and choose a sequence $(t_n)$ in $[0,T]$ such that  $E[Y_{t_n}^2]$ converges to infinity. We can extract a subsequence
$(t_{n_k})$ which converges to some $t\in [0,T]$ and satisfies (by passing to another subsequence, if necessary) $t_{n_k}< t$ or $t_{n_k}>t$ for every $k\in \N$. Then,
 the sequence $(X_{t_{n_k}})$ converges weakly in $(L^2_X)$ to $X_{t-}$ (in the first case) or $X_{t+}$ (in the second case). Hence, by Theorem V.1.1 in \cite{Yosida}, the sequence $(E[Y_{t_{n_k}}^2])_{k\in \N}$ is bounded, a contradiction.

For the converse implication, fix $s\in [0,T)$ and a sequence $(s_n)$, which converges to $s$ from the right. As 
$$
\sup_{n\in \N} E[|Y_{s_n}|^2]\leq \sup_{t\in [0,T]} E[|Y_{t}|^2]<\infty,
$$
there is, by the Banach-Alaoglu theorem \citep[][Theorem V.2.1]{Yosida}, a subsequence $(s_{n_k})$ such that $(Y_{s_{n_k}})$ converges weakly in $(L^2_X)$ to a limit, 
which we denote by $Y_{s+}$. Define, for every $h\in \mathcal{A}$, $\tilde h: [0,T]\rightarrow \R,\; t \mapsto (SY_t)(h)$. As, for every $h \in \mathcal{A}$, $\tilde h$ is 
regulated, we get
$$
\tilde h(s+)=\lim_{k\rightarrow \infty} \tilde h(s_{n_k})=\lim_{k\rightarrow \infty} E[Y_{s_{n_k}} \exp^\diamond(h)]=E[Y_{s+}\exp^\diamond(h)],\quad h\in A.
$$
Now, for every sequence $(t_n)$ converging to $s$ from the right and every $h \in \mathcal{A}$,
$$
E[Y_{t_{n}} \exp^\diamond(h)]=\tilde h(t_n) \rightarrow \tilde h(s+)=E[Y_{s+} \exp^\diamond(h)].
$$
As, moreover, $\sup_{n\in \N} E[|Y_{t_n}|^2]<\infty$, we conclude thanks to Theorem V.1.3 in \cite{Yosida} that
$$
Y_{t_n} \rightarrow Y_{s+}  \quad  \textnormal{ weakly in }(L^2_X).
$$
An analogous argument shows,  for every $s\in (0,T]$, existence of a weak limit $X_{s-}$ in $(L^2_X)$  (as $u$ approaches $s$ from the left).
\end{proof}

The following example is instructive and will be applied in the proof of It\^o's formula.
\begin{exmp}\label{ex:F(X)}
 Suppose that the variance function $V$ of $X$ is regulated and that the Cameron-Martin space $CM_X$ has a dense subset consisting 
of regulated functions. Then, by Corollary \ref{cor:weakly_regulated}, $X$ is weakly regulated. We consider, for some continuous function 
$F:\R\rightarrow\R$, the process $Y_t:=F(X_t)$. Assume that $F$ satisfies the following subexponential growth condition 
\begin{equation}\label{eq:growth}
|F(x)|\leq C e^{ax^2},\quad x\in \mathbb{R},
\end{equation}
for constants $C\geq 0$ and $0\leq a < (4\lambda)^{-1}$ with $\lambda:=\sup_{t\in [0,T]} V(t)$. This standard assumption guarantees that 
$t\mapsto E[Y_t^2]$ is bounded. By \eqref{eq:shift}, the $S$-transform of $Y_t$ is given by
\begin{equation}
 (SF(X_t))(h)= E[F(X_t+\inth(t))]=\psi_F(V(t),\inth(t)),\quad h\in H_X \label{eq:SF(X)},
\end{equation}
where
\begin{equation}\label{eq:psi_F}
\psi_F:[0,\lambda]\times \R \rightarrow \R,\quad (t,x)\mapsto \int_\R F(x+\sqrt{t}y )\frac{1}{\sqrt{2\pi}} e^{-y^2/2}dy.
\end{equation}
Note that $\psi_F(0,x)=F(x)$.

By  \eqref{eq:SF(X)} and Proposition \ref{prop:weakly_regulated}, $Y$ is weakly regulated and the weak limits $Y_{t+}$, $Y_{t-}$ satisfy
$$
(S Y_{t_\pm})(h)= \psi_F(V(t\pm),\inth(t\pm)),\quad h\in H_X.
$$
Define $V^\pm(t)=E[X_{t\pm}^2]$. By Theorem V.1.1 in \cite{Yosida}, $V^\pm(t)\leq V(t\pm)$, and the inequality can be strict, as $X_{t\pm}$ are weak $(L^2_X)$-limits only.
Then,
\begin{equation}
Y_{t+}=\psi_F(V(t+)-V^+(t),X_{t+}),\quad Y_{t-}=\psi_F(V(t-)-V^-(t),X_{t-}),
\end{equation}
since, e.g., for every $h\in H_X$
\begin{eqnarray*}
 && S\left(\psi_F(V(t+)-V^+(t),X_{t+}) \right)(h)\\&=&\int_{\R^2} F(\inth(t+)+\sqrt{V(t+)-V^+(t)}u +\sqrt{V^+(t)}y) \frac{1}{2\pi}e^{-(u^2+y^2)/2} d(u,y)\\ &=&\psi_F(V(t+),\inth(t+)).
\end{eqnarray*}
Thus, the identity $Y_{t+}=F(X_{t+})$ can, in general, only be expected to be valid at those $t\in [0,T)$, for which $V(t+)=V^+(t)$ (which is equivalent to saying that the convergence  from the right to $X_{t+}$ takes place in probability).
\end{exmp}

We next study the set of stochastic discontinuities of $X$.
\begin{defi}
 The process  $X$ is said to be \emph{stochastically continuous} at $t\in [0,T]$, if for every sequence $(t_n)$ in $[0,T]$
$$
t_n \rightarrow t\; \Rightarrow  \; X_{t_n} \rightarrow X_t \textnormal{ in probability}.
$$
We denote by $C_X$ the set of points, at which $X$ is stochastically continuous, and by $D_X=[0,T]\setminus C_X$ the set of stochastic discontinuities of $X$.
\end{defi}

\begin{prop}\label{prop:disc}
 Suppose that the variance function $V$ of $X$ is regulated, $X$ is weakly regulated and $H_X$ is separable. Then the set $D_X$ of stochastic discontinuities of $X$ is at most countable.
\end{prop}
\begin{proof}
 We first apply the separability of $H_X$ and choose a countable dense subset $\mathcal{A}'$ of $H_X$. Let
\begin{eqnarray*}
D&=&\{s\in[0,T];\; V \textnormal{ is discontinuous at }s\} \\ &&\cup \, \left(\bigcup_{h \in \mathcal{A}'} \{s\in[0,T];\; \inth \textnormal{ is discontinuous at }s\}\right).
\end{eqnarray*}
As $V$ and all the $\inth$'s are regulated functions thanks to Proposition \ref{prop:weakly_regulated}, the set $D$ is at most countable. If $t\in [0,T]\setminus D$, 
then we obtain, for every $h\in \mathcal{A}'$
\begin{eqnarray*}
 \lim_{s\rightarrow t} E[X_s h]&=&\lim_{s\rightarrow t} \inth(s)=\inth(t)= E[X_t h] \\
\lim_{s\rightarrow t} E[|X_s|^2]&=&\lim_{s\rightarrow t} V(s)=V(t)= E[|X_t|^2].
\end{eqnarray*}
Now, Theorem V.1.3 in \cite{Yosida} implies that
$$
\lim_{s \rightarrow t}  E[|X_s-X_t|^2]=0,
$$
and, hence, $X$ is stochastically continuous at $t$. In particular, $D_X\subset D$ is at most countable.
\end{proof}
\begin{rem}
 Note that the separability of $H_X$ cannot be dispensed with. Indeed, if $D_X$ is countable, then the rational span of 
$$
\{X_t;\; t\in ([0,T]\cap \mathbb{Q})\cup D_X \}
$$
is a countable dense subset of $H_X$.
\end{rem}

\section{Divergence operator and Skorokhod integral}\label{sec_Malliavin}

In this section, we briefly recall the definition of the Malliavin derivative and the divergence operator, and discuss the notion of  Skorokhod integration.
Our presentation differs from the usual account in the literature, in that we dwell upon the role of the $S$-transform and on regularity properties in the Cameron-Martin space.

Let us first re-interpret the first chaos of $X$ as an isonormal Gaussian process in the sense of Definition 1.1.1 in \cite{Nualart}, parametrized by the Cameron-Martin space of $X$. Then, the \emph{Malliavin derivative} $D$
is a densely defined closed operator from $(L^2_X)$ to $L^2_X(CM_X):=L^2(\Omega, \F^X,P; CM_X)$, which is constructed as follows: 

For a \emph{smooth random variable} $F=f(h_1,\ldots, h_K)$, where $h_1,\ldots h_K \in H_X$
and $F\in \mathcal{C}_b^1(\R^K,\R)$, one defines
$$
DF=\sum_{k=1}^K \frac{\partial F}{\partial x_k} (h_1,\ldots h_K) \, \underline{h}_k.
$$
The Malliavin derivative for smooth random variables is closable from $(L^2_X)$ to $L^2_X(CM_X)$ by Proposition 1.2.1 in \cite{Nualart}. Its domain is usually denoted by $\mathbb{D}_X^{1,2}$. The 
adjoint operator is known as the \emph{divergence operator}. Note that the linear span of the Wick exponentials is dense in $\mathbb{D}_X^{1,2}$, see Theorem 15.110 in \cite{Ja}, and that
$$
D\,\exp^\diamond(h)= \exp^\diamond(h)\,\inth,\quad h\in H_X.
$$
Hence,
$u\in L^2_X(CM_X)$
belongs to the domain $Dom(\delta)$ of the divergence operator $\delta$, if and only if there is a random variable $\delta(u)\in (L^2_X)$ such that, for every $h\in H_X$,
$$
 E[\langle D\exp^{\diamond}(h), u\rangle_{CM_X}]= E[\exp^{\diamond}(h)\delta(u)],
$$
or, equivalently, in terms of the $S$-transform,
\begin{equation*}
 S\left( \langle  u, \inth \rangle_{CM_X}\right)(h)= S(\delta(u))(h),
\end{equation*}
for every $h\in H_X$ (or, by Corollary 3.40 in \cite{Ja}, from a dense subset of $H_X$). The divergence operator $\delta$, defined as above, then is a densely defined closed linear operator from  $L^2_X(CM_X)$ to $(L^2_X)$.
 As the Cameron-Martin space is a function space, our definition of the divergence operator maps stochastic processes to random variables, but it does not yet qualify as an integral operator with respect to $X$. 
Indeed, the minimal property, which such an integral should satisfy is that it maps indicator functions on increments, namely
\begin{equation*}
 {\bf 1}_{(0,t]}\rightarrow X_{t+}-X_{0+},\quad t\in [0,T].
\end{equation*}
However,  in terms of the covariance function $R(t,s):=E[X_t \,X_s]$, we get
$$
\delta(R(t+,\cdot))=X_{t+},
$$
because 
\begin{eqnarray*}
 S\left( \langle  R(t+,\cdot), \inth \rangle_{CM_X}\right)(h)&=& \langle  R(t+,\cdot), \inth \rangle_{CM_X}= E[X_{t+} h]=\inth(t+)\\ &=& (S\,X_{t+})(h).
\end{eqnarray*}
Hence, one can define an integral operator in the following way: Suppose $\mathcal{R}$ is a Banach space of (equivalence classes of) functions from $[0,T]$ to $\R$, such that the indicator functions 
${\bf 1}_{(0,t]}$, $t\in [0,T]$, constitute a total subset of $E$. Suppose $\iota: \mathcal{R} \rightarrow CM_X$ is a continuous linear map such that 
\begin{equation}\label{eq:embedding}
\iota({\bf 1}_{(0,t]})=R(t+,\cdot)-R(0+,\cdot).
\end{equation}
In case of existence, the map $\iota$ is uniquely determined and extends in a canonical way to a continuous linear mapping from the Bochner space $L^2_X(\mathcal{R}):=L^2(\Omega, \F^X,P; \mathcal{R})$ to $L^2_X(CM_X)$.
We may then say that $Z\in L^2_X(\mathcal{R})$ belongs to the domain of the \emph{Skorokhod integral}, if $\iota(Z)\in Dom(\delta)$, and define $\delta_\mathcal{R}(Z):=\delta\circ\iota(Z)$. If we introduce a continuous bilinear map 
on $\mathcal{R}\times CM_X$ via 
$$
\langle r, \inth \rangle_{\mathcal{R},CM_X}:=\langle \iota(r), \inth \rangle_{CM_X},\quad (r,\inth)\in \mathcal{R}\times CM_X,
$$
then we obtain the following $S$-transform characterization of the Skorokhod integral:  $Z\in L^2_X(\mathcal{R})$ belongs to the domain of $\delta_\mathcal{R}$, if and only if there is a $\delta_\mathcal{R}(Z)\in (L^2_X)$ such that for every $h$ from a dense subset of $H_X$,
$$
S\left( \langle  Z, \inth \rangle_{\mathcal{R},CM_X}\right)(h)= S(\delta_\mathcal{R}(Z))(h).
$$ 
Hence, it is crucial to understand the bilinear map $\langle r, \inth \rangle_{\mathcal{R},CM_X}$. Here are some examples.
\begin{exmp}\label{exmp:Skorokhod}
 (i) Suppose that $X$ is an Gaussian martingale with RCLL paths. Then, the variance function $V$ is RCLL and nondecreasing and 
$$
CM_X=\left\{\inth=\int_0^\cdot \dot \inth(s) dV(s);\quad \dot \inth \in L^2([0,T],dV)\right\}, \; \|\inth\|_{CM_X}=\|\dot \inth\|_{L^2(dV)}.
$$
We can now choose $\mathcal{R}=L^2([0,T], dV)$ and $\iota:L^2([0,T], dV)\rightarrow CM_X, \; \dot \inth \mapsto \int_0^\cdot  \dot \inth(s) dV(s)$, which is an isometric isomorphism. Consequently,
$$
\langle r, \inth \rangle_{\mathcal{R},CM_X}=\int_0^T r(s) \dot\inth(s) dV(s) =\int_0^T r(s) d\inth(s).
$$
The operator $\delta_{L^2(dV)}$ then coincides with the classical notion of the Skorokhod 
integral on general measure spaces, see e.g. \citet{Ja}, Chapters 7.2, 15.11, and 16.4.  \\[0.1cm] 
(ii)   Recall that fractional Brownian 
motion is a centered Gaussian process $X$ which has, in dependence of the Hurst parameter $H\in (0,1)$, the covariance function
$$
R(t,s)=\frac{1}{2}\left(t^{2H}+s^{2H}-|t-s|^{2H} \right).
$$
If $H=1/2$, then $X$ is a Brownian motion and we are back in item (i). We now consider the case $H>1/2$. Then every element $\inth$ of the Cameron-Martin space is absolutely continuous with respect to the Lebesgue measure with a density 
$\dot \inth \in L^{1/(1-H)}([0,T],dt)$, see \citet{DU}.
There is, however, no known Hilbert space $\mathcal{R}$ of functions on $[0,T]$ such that $\iota$ in \eqref{eq:embedding} provides an isometric 
isomorphism to $CM_X$, see the discussion in \citet{PT}. In this case, by standard results on fractional integrals, one can choose $\mathcal{R}=L^{1/H}([0,T],dt)$ and obtains, 
$$
\langle r, \inth \rangle_{\mathcal{R},CM_X}=\int_0^T r(s) \dot\inth(s) ds =\int_0^T r(s) d\inth(s).
$$
(iii) We now assume that $X$ is a fractional Brownian motion with Hurst parameter $H<1/2$. Then, there is a Hilbert space of functions $\mathcal{R}$ on $[0,T]$ such that $\iota$ provides an isometric isomorphism on $CM_X$, see \citet{PT}. 
In this case, 
certain elements of the Cameron-Martin space may fail to be absolutely continuous. However, Corollary 2.8 in \citet{BeSPA} or the discussion preceding Definition 3.3 in \citet{CN} can be reformulated in the following way: The subspace of elements $\inth$ of the Cameron-Martin space such that $\inth$ 
has a square-integrable density $\dot\inth$ with respect to the Lebesgue measure and such that for every $r\in \mathcal{R}$
$$
\langle r, \inth \rangle_{\mathcal{R},CM_X}=\int_0^T r(s) \dot\inth(s) ds =\int_0^T r(s) d\inth(s)
$$ 
is dense. Unfortunately, the space $\mathcal{R}$ is unfavorably small and does, with probability 1, not contain the paths of a fractional Brownian motion, if $H\leq 1/4$, see \citet{CN}. This observation
 led to the notion of extended Skorokhod integration, 
see Remark \ref{rem:extendedSk}, (ii),  below.
\end{exmp}

The bottom line of these examples is, that in typical cases, the Skorokhod integral can be characterized by the identity
\begin{equation}\label{eq:SkorkhodS}
\int_0^T (SZ_s)(h) d\inth(s) = S(\delta_\mathcal{R}(Z))(h),
\end{equation}
for elements $h$ from a dense subspace of $H_X$ such that the corresponding element in the Cameron-Martin space $\inth$ is sufficiently smooth. See \citet{MV} for more examples.

\section{Wick-Skorokhod integration}\label{sec_integral}

In this section, we introduce a class of generalized Skorokhod integrals, which is applied throughout the paper. It contains the class of extended Skorokhod integrals, see e.g. \citet{CN, MV, LN}, as special case.

The definition of the integral will be relative to smoothness properties of the elements of the Cameron-Martin space. To this end, let $\mathcal{R}$ denote a vector space  of real-valued functions on $[0,T]$.
\begin{defi}
 We say that the Gaussian process $X$ is $\mathcal{R}$-\emph{dense}, if $CM_X\cap \mathcal{R}$ is a dense subset of $CM_X$. In this case, we denote 
$$
\mathcal{A}_\mathcal{R}:=\{h\in H_X;\; \inth \in \mathcal{R}\},
$$   
which defines a dense subset of the first chaos of $X$. If $CM_X\subset \mathcal{R}$ (and, consequently, $\mathcal{A}_\mathcal{R}=H_X$), we call $X$ $\mathcal{R}$-\emph{regular}.
\end{defi}

\begin{exmp}\label{exmp:regular_dense}
 We denote by $\dot L^p([0,T],\mu)$ the space of functions which are absolutely continuous with respect to the measure $\mu$ with a density in $L^p([0,T],\mu)$, $p\in [1,\infty]$. Then, Example \ref{exmp:Skorokhod} entails the following:
 Any RCLL Gaussian martingale $X$ is  $\dot L^2([0,T],dV)$-regular, fractional Brownian motion with Hurst parameter $H>1/2$ is $\dot L^{1/(1-H)}([0,T],dt)$-regular, and fractional Brownian motion with Hurst parameter 
$H<1/2$ is $\dot L^{2}([0,T],dt)$-dense. Moreover, if $X$ is weakly regulated, then, by Proposition \ref{prop:weakly_regulated}, $X$ is $R([0,T])$-regular, where $R([0,T])$ stands for the space of all regulated functions 
on $[0,T]$. Trivially, any Gaussian process $X$ is $CM_X$-regular.
\end{exmp}

In view of \eqref{eq:SkorkhodS}, we propose the following definition:
\begin{defi}
Suppose $X$ is  $\mathcal{R}$-dense.
 A process $Z:[0,T]\rightarrow (L^2_X)$ is said to be $\mathcal{R}$-\emph{Wick-Skorokhod integrable} with respect to $X$, if for every 
$h\in \mathcal{A}_\mathcal{R}$ the integral $\int_0^T (SZ_s)(h)\,d\inth(s)$ exists (in the sense of Henstock and Kurzweil) and if there is a random variable $\int_0^T Z_s\, d^{\diamond}_{\mathcal{R}} X_s \in (L^2_X)$ 
such that for every $h\in \mathcal{A}_\mathcal{R}$
$$
S\left( \int_0^T Z_s\, d^{\diamond}_{\mathcal{R}} X_s\right)(h)= \int_0^T (SZ_s)(h)\,d\inth(s).
$$
We then call $\int_0^T Z_s\, d^{\diamond}_{\mathcal{R}} X_s$ the $\mathcal{R}$-\emph{Wick-Skorokhod integral} of $Z$ with respect to $X$.
\end{defi}
Note that, if $X$ is  $\mathcal{R}$-regular, then the $S$-transform characterization of the $\mathcal{R}$-\emph{Wick-Skorokhod integral} in the above definition holds for every element $h$ in the first chaos, and not only on a dense subspace.

We sometimes write $\int_{0+}^T Z_s\, d^{\diamond}_{\mathcal{R}} X_s:=\int_{0}^T  {\bf 1 }_{(0,T]}(s)Z_s\, d^{\diamond}_{\mathcal{R}} X_s$, provided the integral on the right-hand side exists.
The same notation will be applied for other types of integrals, when integration is understood over $(0,T]$ rather than $[0,T]$.

\begin{rem}[Henstock-Kurzweil integral]\label{rem:HK}
 The Henstock-Kurzweil integral can be applied to give a meaning to integrals of the form $\int_0^T u(s)dr(s)$ for suitably pairs of functions $u,r:[0,T]\rightarrow \R$, without assuming that 
the integrator $r$ is of bounded variation. We briefly recall the construction: A \emph{gauge function} is any function $\delta:[0,T]\rightarrow (0,\infty)$. A tagged partition 
$\tau:=\{([s_{i-1},s_i],y_i);\;i=1,\ldots,n\}$ of the interval $[0,T]$ is called \emph{$\delta$-fine}, if $y_i-\delta(y_i)\leq s_{i-1}\leq y_i \leq s_{i} \leq y_i+\delta(y_i)$ for every $i=1,\ldots,n$. 
The \emph{Riemann sum} for the integral  $\int_0^T u(s)dr(s)$ with respect to the tagged partition $\tau$ is given by 
$$
S_{RS}(u,r,\tau):=\sum_{i=1}^n u(y_i)(r(s_i)-r(s_{i-1})).
$$
If there is an $I\in \R$ such that for every $\epsilon>0$ there is a gauge function $\delta$ such that $|S_{RS}(u,r,\tau)-I|<\epsilon$ for every $\delta$-fine tagged partition $\tau$, then 
$I$ is uniquely determined, denoted by $\int_0^T u(s) dr(s)$ and called the \emph{Henstock-Kurzweil integral} of $u$ with respect to $r$. A brief review of the relation between 
the Henstock-Kurzweil integral and other Stieltjes-type integrals can be found in Appendix F of Part I in \cite{DN}. We just note the following important relation to the Lebesgue-Stieltjes integral:
If $r$ is of bounded variation, then $r$ uniquely determines a signed measure $\mu_r$ via the relation
$$
\mu_r([0,t]):=r(t+)-r(0), \; 0\leq t <T,\quad \mu_r([0,T])=r(T)-r(0).
$$ 
If the Lebesgue-Stieltjes integral $\int_{0}^T u(s) \mu_r(ds)$ exists, then so does the Hen\-stock-Kurzweil integral $\int_0^T u(s)dr(s)$ and both integrals coincide.
\end{rem}

\begin{rem}\label{rem:extendedSk}
(i) The choice of the function space $\mathcal{R}$ reflects the following tradeoff: The larger the space $\mathcal{R}$, the larger the class of integrators $X$, for which the 
$\mathcal{R}$-Skorokhod integral can be defined. However, a large function space $\mathcal{R}$ means that the members of $\mathcal{R}$ may lack good smoothness properties, and, hence,
the space of integrands may become small. More precisely, if $\mathcal{R}_1\subset \mathcal{R}_2$, then any $\mathcal{R}_1$-dense Gaussian process is also  $\mathcal{R}_2$-dense. Moreoever, if $X$ is $\mathcal{R}_1$-dense, then the 
$\mathcal{R}_1$-Wick-Skorokhod integral is an extension of the $\mathcal{R}_2$-Wick-Skorokhod integral.
\\[0.1cm]
(ii) The notion of an extended Skorokhod integral in the sense of \cite{CN,MV,LN} can be cast into our setting: Indeed, in all these cases, $\mathcal{R}$ can be chosen as a suitable subspace 
of $\dot L^p([0,T],dt)$ for some fixed $p\in (1,\infty)$. It is important to note that some of these references impose integrability conditions 
on the paths of the  integrand $Z$, e.g., $Z\in L^2_X(L^2([0,T]))$ in \cite{MV} (where $p=2$). In our definition, we merely require that $\int_0^T (SZ_s)(h),d\inth(s)$ exists 
for $h\in \mathcal{A}_\mathcal{R}$. As we will work under a much weaker regularity condition than absolute continuity for the elements of the Cameron-Martin space for most of the paper, this relaxation of the integrability conditions 
 turns out to be crucial to ensure e.g. integrability of $X$ with respect to itself.
\end{rem}

The next example, which discusses the integrability of step processes, clarifies the relation of our Skorokhod type integral to Wick products:
\begin{exmp}\label{exmp:simple}
 Suppose  $X$ is weakly regulated, and, thus, every element $\inth\in CM_X$ is a regulated function.  We say, $Z$ is a \emph{simple integrand}, if it is of the form  
$$
Z:= F_0 {\bf 1}_{\{0\}}+\sum_{i=1}^n \left( G_i {\bf 1}_{(t_{i-1},t_i)} + F_i {\bf 1}_{\{t_i\}} \right),\quad 
$$
$0=t_0<t_1<\cdots<t_n=T,\; F_i,G_i \in \mathbb{D}^{1,2}_X$. Then, $t\mapsto (SZ_t)(h)$ is a real-valued step function and, because  $\inth$ is a regulated function for every $h\in H_X$, the 
Henstock-Kurweil integral $\int_0^T  (SZ_t)(h) d\inth(s)$ exists. Indeed, we get, for every $h\in H_X$ (applying the jump convention $\Delta\inth(T)=\Delta\inth^-(T)$),
\begin{eqnarray}\label{eq:S_simple_integrand}
&& \int_0^T (SZ_s)(h)\, d\inth(s) =
 (SF_0)(h)\cdot \Delta^+\inth(0)  \\ & +&\sum_{i=1}^n  (SG_i)(h)\left(\inth(t_i-)-\inth(t_{i-1}+)\right)+
(SF_i)(h)\cdot \Delta\inth(t_i).\nonumber
\end{eqnarray}
Let us now recall that two random variables $\eta, \xi \in (L^2_X)$ are said to have a \emph{Wick product}, if there is a random variable $\eta \diamond \xi \in (L^2_X)$ such that
$$
S(\eta \diamond \xi)(h)=(S\eta)(h)(S\xi)(h)
$$
for every $h\in H_X$. By Proposition 1.3.3 in \cite{Nualart}, $F\diamond g$ exists for every $F\in \mathbb{D}^{1,2}_X$, $g\in H_X$ and relates to the ordinary product via
$$
F\diamond g= Fg-\langle DF, \intg\rangle_{CM_X},\quad F\in \mathbb{D}^{1,2}_X,\;g\in H_X.
$$
In view of \eqref{eq:Stransform1} and \eqref{eq:S_simple_integrand}, we, thus, obtain for every $h\in H_X$,
\begin{eqnarray*}
 && \int_0^T (SZ_s)(h)\, d\inth(s)\\  &=& S\left( F_0\diamond(\Delta^+X_0)+\sum_{i=1}^n \left(G_i\diamond (X_{t_{i}-}-X_{t_{i-1}+}) + F_i \diamond (\Delta X_{t_{i}})\right)\right)(h)
\end{eqnarray*}
Hence,
\begin{equation*}%\label{eq:Skorokhod_simple}
 \int_0^T Z_s \, d_\mathcal{R}^\diamond X_s= F_0\diamond(\Delta^+X_0)+\sum_{i=1}^n \left(G_i\diamond (X_{t_{i}-}-X_{t_{i-1}+}) + F_i \diamond (\Delta X_{t_{i}})\right), 
\end{equation*}
for any function space $\mathcal{R}$ such that $X$ is $\mathcal{R}$-dense, e.g. $\mathcal{R}=R([0,T])$. Hence, the Wick-Skorokhod integral of a simple integrand is a Young-Stieltjes sum involving Wick products, as expected for Skorokhod type integrals.  
\end{exmp}

We next relate our notion of Wick-Skorokhod integration to the usual stochastic It\^o integral \citep[see e.g.][]{Protter} in the martingale case:
\begin{thm}\label{thm:relation_ito}
Let $\mathbb{F}^X=(\F^X_t)_{t\in [0,T]}$ denote the augmentation of the filtration generated by $X$. Suppose $X$ is an $\mathbb{F}^X$-martingale with RCLL paths and variance function $V$. 
 If $Z$ is $\mathbb{F}^X$-predictable and satisfies 
$$
E\left[\int_0^T |Z_s|^2 dV(s)\right]<\infty,
$$
then the {Wick-Skorokhod integral} $\int_0^T Z_s\, d^{\diamond}_{\mathcal{R}} X_s$  exists for $\mathcal{R}=\dot L^2([0,T],dV)$ and coincides with the stochastic It\^o integral $\int_{0+}^T Z_s dX_s$. 
\end{thm}
\begin{rem}
The statement of the previous theorem cannot be extended to semimartingales, because Skorokhod integration is based on the Wick product. Indeed, suppose that $X$ is a Gaussian RCLL semimartingale. If $X$ is not an $\mathbb{F}^X$-martingale, then we find 
$0\leq a<b<c<d\leq T$ such that $C:=E[(X_b-X_a)(X_d-X_c)]\neq 0$.  Consider the simple predictable process $Z_t=(X_b-X_a){\bf 1}_{(c,d]}(t)$. Then, by Example \ref{exmp:simple},
\begin{eqnarray*}
\int_0^T Z_s \, d_{R([0,T])}^\diamond X_s&=&(X_b-X_a)\diamond (X_d-X_c)= (X_b-X_a)(X_d-X_c)-C \\ &\neq& (X_b-X_a)(X_d-X_c),
\end{eqnarray*}
where the right-hand side equals the It\^o integral of $Z$ with respect to $X$.
\end{rem}
\begin{proof}[Proof of Theorem \ref{thm:relation_ito}]
Note first that, by the martingale property of $X$, the set  $D_X$ of stochastic discontinuties of $X$ consists of the jump times of the nondecreasing RCLL variance 
function $V$ of $X$, and is thus at most countable.
 By Theorem 1.8 in \cite{JM}, 
$$
X_t=X^c_t+\sum_{s\in D_X\cap (0,t]} (X_{s}-X_{s-}),
$$
where $X^c$ is a Gaussian martingale with continuous paths. 
Of course, in this decomposition, the pathwise left limits of the RCLL martingale $X$ coincide, for every $s\in(0,T]$, $P$-almost surely with the weak $(L^2_X)$-limits $X_{s-}$.
We write $V^c$ and $V^d$ for the continuous and the discrete part of $V$. Then, 
 $V^c$ is the quadratic variation of $X^c$, $V^d$ is the predictable 
compensator of $\sum_{s\in D_X\cap (0,\cdot]} (X_{s}-X_{s-})^2$, and, hence, $V$ is the predictable compensator of the quadratic variation $[X]$ of $X$. Thus, if $Z$ is predictable and satisfies 
the assumed integrability condition, then, by the isometry of the stochastic It\^o integral, $\int_{0+}^T Z_s dX_s$ exists in $(L^2_X)$ and satisfies 
$$
E\left[ \left|\int_{0+}^T Z_s dX_s\right|^2\right]=E\left[\int_{0+}^T |Z_s|^2 d[X]_s\right]=E\left[\int_{0+}^T |Z_s|^2 dV(s)\right].
$$
In particular, the first chaos of $X$ can be represented as 
$$
H_X=\left\{aX_0+\int_{0+}^T \h(s) dX_s;\quad a\in \R, \h\in L^2([0,T], dV) \right\}.
$$
We now fix a generic element $h=aX_0+\int_{0+}^T \h(s) dX_s$ of the first chaos of $X$ and define 
\begin{eqnarray*}
\mathcal{E}_t&:=&\exp\left\{aX_0 -\frac{a^2}{2}  V(0)\right\}\exp\left\{\int_{0+}^t \h(s)dX_s -\frac{1}{2}\int_{0+}^t |\h(s)|^2 dV(s) \right\},\\ Y_t&:=&\int_{0+}^t Z_s dX_s.
\end{eqnarray*}
Note that $\inth(t)=aV(0)+\int_{0+}^t \h(s)dV(s)$. Thus, in view of Remark \ref{rem:HK}, all we need to show is that
\begin{equation}\label{eq:relation_ito}
 E[\mathcal{E}_T Y_T] = \int_{0+}^T E[\mathcal{E}_T Z_s] \h(s) dV(s),
\end{equation}
 where the integral on the right-hand side is a Lebesgue-Stieltjes integral.
Applying It\^o's formula for semimartingales \citep[][Theorem II.32]{Protter}, we obtain after some  elementary manipulations
$$
 \mathcal{E}_t=\mathcal{E}_0+\int_{0+}^t \mathcal{E}_{s-} \h(s) d X^c_s +\sum_{s\in D_X\cap(0,t]} \mathcal{E}_{s-}\left(e^{\h(s)\Delta X_s-\frac{1}{2} \h(s)^2\Delta V(s)}-1\right).
$$
Integration by parts (for semimartingales) thus yields
\begin{equation}\label{eq:hilf0001}
 \mathcal{E}_t Y_t= \int_{0+}^t \mathcal{E}_{s-} dY_s + \int_{0+}^t {Y}_{s-} d\mathcal{E}_s + [\mathcal{E},Y]_t
\end{equation}
where the quadratic covariation of $\mathcal{E}$ and $Y$ is given by
\begin{eqnarray} \label{eq:hilf0002}
[\mathcal{E},Y]_t&=& \int_{0+}^t  Z_{s} \mathcal{E}_{s-} \h(s) dV^c(s) \nonumber \\ &&+ \sum_{s\in D_X\cap(0,t]} Z_{s}\mathcal{E}_{s-}\left(e^{\h(s)\Delta X_s-\frac{1}{2} \h(s)^2\Delta V(s)}-1\right)\Delta X_s.
\end{eqnarray}
Since $\mathcal{E}$ and $Y$ are square-integrable RCLL martingales, we may conclude from Emery's inequality \citep[][Theorem V.3]{Protter} that both stochastic integrals 
in \eqref{eq:hilf0001} are martingales (of class $\mathcal{H}^1$) and, consequently, have zero expectation. Taking into account, that, for every $s\in D_X$, the jumps sizes
 $\Delta X_s$ 
are independent of $\mathcal{F}^X_{s-}$ and $Z_s$ is $\mathcal{F}^X_{s-}$-measurable by predictability, we obtain, thanks to \eqref{eq:hilf0001}-\eqref{eq:hilf0002},
\begin{eqnarray*}
 E[\mathcal{E}_T Y_T] &=& 
 \int_{0+}^T  E[Z_{s} \mathcal{E}_{s-}] \h(s) dV^c(s) \\ && + \sum_{s\in D_X\cap(0,T]} E[Z_{s}\mathcal{E}_{s-}]E\left[\left(e^{\h(s)\Delta X_s-\frac{1}{2} \h(s)^2\Delta V(s)}-1\right)\Delta X_s\right] \\
\end{eqnarray*}
As, by \eqref{eq:Stransform1},
$$
E\left[\left(e^{\h(s)\Delta X_s-\frac{1}{2} \h(s)^2\Delta V(s)}-1\right)\Delta X_s\right]=\h(s) \Delta V(s),
$$ 
we arrive at
$$
E[\mathcal{E}_T Y_T] =\int_{0+}^T  E[Z_{s} \mathcal{E}_{s-}] \h(s) dV(s).
$$
Now, since $Z_s$ is $\F^X_{s-}$-measurable and $\mathcal{E}$ is a martingale, we finally obtain
$$
[Z_{s} \mathcal{E}_{s-}]=E[Z_{s} E[\mathcal{E}_{T}|\F^X_{s-}]]=E[Z_s \mathcal{E}_{T}],
$$
which yields \eqref{eq:relation_ito}, and finishes the proof.
\end{proof}

For the statement and proof of the It\^o formula, we formulate the regularity requirement of elements in the Cameron-Martin space in terms of the quadratic variation as follows: We denote by $W^*_2=W^*_2([0,T])$ the subspace of 
regulated functions $u:[0,T]\rightarrow \R$ such that 
$$
\sigma_2(u):=\sum_{s\in (0,T]} |\Delta^- u(s)|^2+\sum_{s\in [0,T)} |\Delta^+ u(s)|^2<\infty 
$$
and such that, for every $\epsilon>0$, there is a partition $\lambda$ of $[0,T]$ such that for all refinements $\kappa=\{0=s_0<s_1<\cdots<s_n\leq T\}$ of $\lambda$
$$
\left|\left(\sum_{i=1}^n |u(s_i)-u(s_{i-1})|^2\right)- \sigma_2(u) \right|<\epsilon.
$$ 
Roughly speaking, this requirement means that the continuous part of the quadratic variation (in the sense of stochastic analysis) must vanish. 

The following structural assumption on the Cameron-Martin space will be assumed for the main results of this paper:
\begin{itemize}
 \item[{\bf (H)}] The Cameron-Martin space $CM_X$ of $X$ is separable and $X$ is  $W^*_2$-dense.
\end{itemize}
 Under condition $(H)$, we can thus study the Wick-Skorokhod integral for $\mathcal{R}= {W}^*_2$. In order to lighten the notation, we skip the subscript from the notation of the Wick-Skorokhod integral 
in this case and simply write
$$
\int_0^T Z_s\, d^{\diamond} X_s:=\int_0^T Z_s\, d^{\diamond}_{{W}^*_2} X_s.
$$
\begin{rem}\label{rem:Wp}
(i) Recall that, for $p\geq 1$, the $p$-variation of a  regulated function $u:[0,T]\rightarrow \R$ is defined to be
\begin{eqnarray*}
v_p(u)&:=&\sup\Bigl\{ \sum_{j=1}^m |u(t_j)-u(t_{j-1})|^p;\\ && \quad \quad \; m\in \N,\; 0=t_0<t_1<\cdots<t_{m-1}<t_m= T\Bigr\}.  
\end{eqnarray*}
A  regulated function $u$ is said to belong to $W_p([0,T])$, if $v_p(u)<\infty$. By Lemmas II.2.3 and II.2.14 in \cite{DN}, $W_p([0,T])\subset W_2^*([0,T])\subset W_2([0,T])$ 
for every $1\leq p <2$. With this notation,  $W_1([0,T])$ is the space of bounded variation functions on $[0,T]$. \\[0.1cm] 
(ii)  Assumption (H) will be discussed in some more detail in Section \ref{sec_reg} below. We note that it is obviously satisfied when $X$ has RCLL paths (or, less restrictively, 
is stochastically RCLL) and, for every fixed
$s\in [0,T]$, the covariance function $R(t,s):=E[X_tX_s]$ is of bounded variation as function in $t$. This already covers a large class of relevant Gaussian processes. \\[0.1cm] 
(iii) It is  important, that the regularity requirement is imposed on the elements of the Cameron-Martin space and not on the paths of the process $X$. For instance, in the fractional Brownian motion case, $X$ is 
$W^*_2$-regular for every choice of the Hurst parameter $H\in (0,1)$, and then, in particular, satisfies (H). For $H\geq 1/2$, this follows from item (i) and Example \ref{exmp:regular_dense}, while for $H<1/2$ the asserted regularity of the Cameron-Martin space 
is implied by item (i) in conjunction with Theorem 1.1 and Example 2.8 in \citet{Fal}.
 However, it is well-known that almost every path of the fractional Brownian motion belongs to 
${W}^*_2$, if and only if $H>1/2$. 
\end{rem}

\section{It\^o's formula} \label{sec_ito}

After these preparations on Wick-Skorokhod integration and weakly regulated processes, we can now  state and prove a
general Gaussian It\^o formula in the presence of stochastic discontinuities.

Let us first recall that $V^{\pm}(t)$ denotes the variance of $X_{t\pm}$ and that, for a sufficiently integrable function $F$, $\psi_F$ is given by the convolution
$$
\psi_F(t,x):=\int_\R F(x+\sqrt{t}y )\frac{1}{\sqrt{2\pi}} e^{-y^2/2}dy,
$$
see \eqref{eq:psi_F}.
\begin{thm}\label{thm:Ito}
 Suppose $X$ is a centered Gaussian process  satisfying (H), the variance function $V$ of $X$ is of bounded variation and
\begin{eqnarray}\label{eq:ito_ass}
&& \sum_{s\in D_X\cap [0,T)} \left(E[(\Delta^+ X_s)^2]+(V(s+)-V^+(s)) \right) \\ &+& \sum_{s\in D_X\cap (0,T]}  \left( E[(\Delta^- X_s)^2]+(V(s-)-V^-(s))\right)  <\infty.\nonumber 
\end{eqnarray}
Assume $F\in \mathcal{C}^2(\mathbb{R})$ and $F,F',F''$ satisfy the growth condition \eqref{eq:growth} with $\lambda=\sup_{t\in [0,T]} V(t)$. 
Then, $\int_0^T F'(X_s) d^\diamond X_s$ exists and the following It\^o formula holds in $(L^2_X)$:
\begin{eqnarray*}
 &&F(X_T)-F(X_0)\\ &=&\int_0^T F'(X_s) d^\diamond X_s + \frac{1}{2} \int_0^T F''(X_s) dV(s) \\
&& + \sum_{s\in D_X \cap (0,T]} \Bigl( F(X_s)- \psi_F(V(s-)-V^-(s), X_{s-})-F'(X_s) \Delta^-X_s \\
&& \quad\quad\quad\quad\quad + \frac{1}{2} F''(X_s)(E[(\Delta^-X_s)^2] +V(s-)-V^-(s)) \Bigr) \\ 
 &&+ \sum_{s\in D_X \cap [0,T)} \Bigl( \psi_F(V(s+)-V^+(s), X_{s+})-F(X_s) -F'(X_s) \Delta^+X_s \\
&& \quad\quad\quad\quad\quad - \frac{1}{2} F''(X_s)(E[(\Delta^+X_s)^2] +V(s+)-V^+(s))\Bigr).
\end{eqnarray*}
Here, the set $D_X$ of stochastic discontinuities of $X$ is at most countable and both sums converge absolutely in $(L^2_X)$.
\end{thm}
Note first that, under the assumptions of Theorem \ref{thm:Ito}, $X$ is weakly regulated by Corollary \ref{cor:weakly_regulated} and has at most countably many stochastic 
discontinuities by Proposition \ref{prop:disc}.

We prove Theorem \ref{thm:Ito} by an $S$-transform approach, which originates in the work by \citet{Kubo} on It\^o's formula 
for generalized functionals of a Brownian motion in the setting of white noise analysis. It has since then be succesfully applied to wider classes 
of (stochastically) continuous Gaussian processes, see e.g. \citet{BeSPA, BeBernoulli, AK, LV, Leb}. In a closely related development, the Malliavin calculus approach to 
It\^o's formula, replaces the $S$-transform (and, thus, the pairing with  Wick exponentials) by    pairings with Hermite polynomials,
see \citet{AMN,MV,KRV,LN}.

The key idea of the $S$-transform approach to It\^o's formula is the following: While the process $F(X_t)$ may lack good path regularity, its $S$-transform $t\mapsto  
S(F(X_t))(\eta)$ is typically more regular and may be expanded  via a `classical' chain rule. In a second step, the resulting terms, which appear after application of the chain rule, must be identified
as the $S$-transforms of the different terms in the Gaussian It\^o formula.  
The main contribution to this technique of proof in the present paper is to deal with the jumps that occur 
at the times of stochastic discontinuities and to make this technique applicable under the very weak regularity assumption (H) on the Cameron-Martin space. 

Expanding the $S$-transform of  $F(X_t)$ via the chain rule in Theorem \ref{thm:chain_rule} for the Henstock-Kurzweil integral leads to the following 
result.
\begin{prop}\label{prop:ito}
Let all assumptions of Theorem \ref{thm:Ito} be in force. Then, for every $h\in \mathcal{A}_{W^*_2}$, 
$\int_0^T (S\, F'(X_s))(h)\,d\inth(s)$ exists as Henstock-Kurzweil integral and 
\begin{eqnarray*}
 &&(S\,F(X_T))(h)\\&=& (S\,F(X_0))(h) + \int_0^T (S\, F'(X_s))(h)\,d\inth(s)+ \frac{1}{2} \int_0^T (S\, F''(X_s))(h)\,dV(s) \\ 
&& +  \sum_{s\in D_X \cap (0,T]} \Bigl((S\,F(X_s))(h)- S(\,\psi_F(V(s-)-V^-(s), X_{s-}))(h) \\ && \quad\quad-(S\,F'(X_s)) \Delta^-\inth(s) -
\frac{1}{2} (S\,F''(X_s)) \Delta^-V(s) \Bigr)  \\ 
&& +  \sum_{s\in D_X \cap [0,T)} \Bigl(S(\,\psi_F(V(s+)-V^+(s), X_{s+}))(h) -(S\,F(X_s))(h) \\ && \quad\quad-(S\,F'(X_s)) \Delta^+\inth(s) -
\frac{1}{2} (S\,F''(X_s)) \Delta^+V(s) \Bigr),
\end{eqnarray*}
where the two sums converge absolutely.
\end{prop}
\begin{proof}
 Recall that, by \eqref{eq:SF(X)},
$$
(SF(X_t))(h)= \psi_F(V(t),\inth(t)).
$$
So, we first check that $\psi_F$ satisfies the regularity requirements of Theorem \ref{thm:chain_rule}. By the subexponential growth condition \eqref{eq:growth}, we can interchange differentiation and integration 
and obtain, for every $x\in \R$, $t\in [0,\lambda]$
\begin{equation}\label{eq:heat1}
\frac{\partial}{\partial x} \psi_F(t,x)= \psi_{F'}(t,x),\quad \frac{\partial^2}{\partial x^2} \psi_F(t,x)= \psi_{F''}(t,x).
\end{equation}
As $\psi_{F''}$ is continuous on $[0,\lambda]\times \R$, the Lipschitz condition in the first line of \eqref{eq:reg_chain_rule} is clearly satisfied.
Moreover, for $t\in (0,\lambda]$ and $x\in \R$, integration by parts yields
\begin{equation}\label{eq:heat2}
\frac{\partial}{\partial t} \psi_F(t,x)= \frac{1}{2} \int_\R F'(x+\sqrt{t}y )\frac{y}{\sqrt{2\pi t}} e^{-y^2/2}dy=\frac{1}{2}\psi_{F''}(t,x).
\end{equation}
This identity is also valid at $t=0$, because for every $\epsilon>0$, by a Taylor expansion,
\begin{eqnarray*}
 && \frac{\psi_F(\epsilon,x)-\psi_F(0,x)}{\epsilon}\\&=&  \int_\R F'(x)\frac{y}{\sqrt{2\pi \epsilon}} e^{-y^2/2}dy + \frac{1}{2} \int_\R F''(x)\frac{y^2}{\sqrt{2\pi}} e^{-y^2/2}dy+R_x(\epsilon) \\
&=&  \frac{1}{2} F''(x)+ R_x(\epsilon)
\end{eqnarray*}
with remainder term
$$
R_x(\epsilon)= \int_\R  \int_0^1 (1-v)(F''(x+v\sqrt{\epsilon} y) -F''(x)) dv  \frac{y^2}{\sqrt{2\pi}} e^{-y^2/2}dy,
$$
and by dominated convergence $R_x(\epsilon)$ tends to zero, as $\epsilon$ goes to zero, for every $x\in \R$. 
In order to show the H\"older-type condition in the second line of \eqref{eq:reg_chain_rule}, we define 
$$
K:  \R \times [0,\lambda]\times [0,\lambda]\rightarrow \R_{\geq 0},\quad (x,t,s)\mapsto\left\{\begin{array}{cl} \frac{| \psi_{F'}(t,x)-\psi_{F'}(s,x)|}{|t-s|^{1/2}}, & t\neq s, \\ 0, & t=s. \end{array} \right.
$$
Then, by \eqref{eq:heat1},
$$
\left|\frac{\partial}{\partial x} \psi_F(t,x)-\frac{\partial}{\partial x} \psi_F(s,x)\right|=K(x,t,s) |t-s|^{1/2},
$$
and it suffices to show that $K$ is continuous. However,
\begin{eqnarray*}
 &&\psi_{F'}(t,x)-\psi_{F'}(s,x)\\ &=& \int_\R \int_{x+\sqrt{s}y}^{x+\sqrt{t}y} F''(r)dr  \frac{1}{\sqrt{2\pi }} e^{-y^2/2}dy 
= (\sqrt{t}-\sqrt{s}) \\ &&\times  \int_\R \int_{0}^{1} \left( F''(x+y(\sqrt{s}(1-v)+v\sqrt{t}))-F''(x)\right)dv  \frac{y}{\sqrt{2\pi }} e^{-y^2/2}dy.
\end{eqnarray*}
Thus, for $t\neq s$,
\begin{eqnarray*}
 && K(x,t,s) = \frac{|t-s|^{1/2}}{\sqrt{t}+\sqrt{s}}\\ &&\times  \left|\int_\R \int_{0}^{1} \left( F''(x+y(\sqrt{s}(1-v)+v\sqrt{t}))-F''(x)\right)dv  \frac{y}{\sqrt{2\pi }} e^{-y^2/2}dy\right|,
\end{eqnarray*}
which, by dominated convergence, implies that $K$ is continuous at every $(x_0,t_0,s_0)\in\R\times([0,\lambda]^2\setminus \{(0,0)\})$. In order to show continuity  
at $(x_0,0,0)$ for $x_0\in \R$, let $(x_n,t_n,s_n)$ be as sequence converging to  $(x_0,0,0)$. Then,
\begin{eqnarray*}
 && |K(x_n,t_n,s_n)| \\ &\leq& \int_\R \int_{0}^{1} \left| F''(x_n+y(\sqrt{s_n}(1-v)+v\sqrt{t_n}))-F''(x_n)\right|dv  \frac{|y|}{\sqrt{2\pi }} e^{-y^2/2}dy,
\end{eqnarray*}
which tends to zero by dominated convergence.

We can, thus, apply the chain rule in Theorem \ref{thm:chain_rule} to $\psi_F(V(t),\inth(t))$ for $h\in \mathcal{A}_{W^*_2}$, and obtain, in view of 
\eqref{eq:heat1}--\eqref{eq:heat2},
\begin{eqnarray*}
&& \psi_F(V(T),\inth(T))-\psi_F(V(0),\inth(0))\\ &=& \int_0^T \psi_{F'}(V(s),\inth(s)) \, d\inth (s)+ \frac{1}{2}
\int_0^T \psi_{F''}(V(s),\inth(s))\,dV(s) \\ &&+
 \sum_{s\in (0,T]}  \psi_F(V(s),\inth(s))- \psi_{F}(V(s-),\inth(s-))-  \psi_{F'}(V(s),\inth(s)) \Delta^-\inth(s)\\ &&\quad \quad-  \frac{1}{2} \psi_{F''}(V(s),\inth(s))\Delta^-V(s) \\ 
&&+
 \sum_{s\in [0,T)}  \psi_F(V(s+),\inth(s+))- \psi_F(V(s),\inth(s))-   \psi_{F'}(V(s),\inth(s)) \Delta^+\inth(s) \\ &&  \quad \quad -\frac{1}{2} \psi_{F''}(V(s),\inth(s))\Delta^+V(s),
\end{eqnarray*}
including the existence of the integral with respect to $\inth$ as Henstock-Kurzweil integral and the absolute convergence of the two sums. As $\inth$ and $V$ are continuous at $s$, if 
$X$ is stochastically continuous at $s$, the two sums can be restricted to  $(0,T]\cap D_X$ and $[0,T)\cap D_X$, respectively, without changing their values. Finally, each term 
in the above identity can be rewritten in terms of the $S$-transform by an application of Example \ref{ex:F(X)}, which yields the assertion.
\end{proof}

We are now in the position to prove the It\^o formula in Theorem \ref{thm:Ito}.

\begin{proof}[Proof of Theorem \ref{thm:Ito}]
 We first treat the term, which involves the jumps from the right. For $s\in D_X \cap [0,T)$, we consider
\begin{eqnarray*}
 J_s^+&:=&\psi_F(V(s+)-V^+(s), X_{s+})-F(X_s) -F'(X_s) \Delta^+X_s \\ && - \frac{1}{2} F''(X_s)(E[(\Delta^+X_s)^2]  +V(s+)-V^+(s)).
\end{eqnarray*}
Note that the subexponential growth condition \eqref{eq:growth} ensures that each $J_s^+$ belongs to $(L^2_X)$. In  order to compute the $S$-transform of $J^+_s$, we note that,
for every $h\in H_X$,
\begin{eqnarray}\label{eq:Stransform_derivative}
 && S(F'(X_s) \Delta^+X_s)(h)= E[ F'(X_s+\inth(s))( \Delta^+X_s+ \Delta^+ \inth(s))] \\
&=& E[F'(X_s+\inth(s)) X_s] E[X_s\Delta^+X_s]/V(s) + E[ F'(X_s+\inth(s))]\Delta^+\inth(s) \nonumber\\
&=& E[F''(X_s+\inth(s))]E[X_s\Delta^+X_s] + E[ F'(X_s+\inth(s))]\Delta^+\inth(s)\nonumber\\ &=& S(F''(X_s))(h)E[X_s\Delta^+X_s] + S(F'(X_s))(h)\Delta^+\inth(s). \nonumber
\end{eqnarray}
 Here, the first equality is due to \eqref{eq:shift}, the second one follows by projecting $\Delta^+X_s$ on $\{y X_s;\; y \in \R\}$, the third one by 
rewriting the expectation as integral with respect to the Gaussian density and by integration by parts.  Note that this well-known identity can alternatively be derived from the relation between 
Wick product and Malliavin derivative, see \citet[Proposition 1.3.4]{Nualart}.
As 
\begin{equation}\label{eq:hilf0003}
2E[X_s\Delta^+X_s]+E[(\Delta^+X_s)^2]  +V(s+)-V^+(s)=V(s+)-V(s)=\Delta^+V(s),
\end{equation}
we obtain, for every $h\in H_X$
\begin{eqnarray}\label{eq:Stransform_jumps}
S( J_s^+)(h)&=&S(\,\psi_F(V(s+)-V^+(s), X_{s+}))(h) -(S\,F(X_s))(h)  \\ &&-(S\,F'(X_s)) \Delta^+\inth(s) -
\frac{1}{2} (S\,F''(X_s)) \Delta^+V(s).\nonumber
\end{eqnarray}
%Hence,
%\begin{equation}
% \sum_{s\in  D_X \cap [0,T)} |S( J_s^+)(h)|<\infty
%\end{equation}
%for every $h\in \mathcal{A}_{W^*_2}$ by Proposition \ref{prop:ito}. 
We next show that the sum of $J^+_s$ converges absolutely in $(L^2_X)$ as $s$ runs through $D_X \cap [0,T)$.
To this end, we decompose $J_{s}^+= J_{s}^{+,1} - J_{s}^{+,2}$, where 
\begin{eqnarray*}
 J_s^{+,1}&:=&\psi_F(V(s+)-V^+(s), X_{s+})-F(X_s) -F'(X_s) \Delta^+X_s, \\
J_s^{+,2}&:=&\frac{1}{2} F''(X_s)(E[(\Delta^+X_s)^2]  +V(s+)-V^+(s)).
\end{eqnarray*}
Taylor's theorem yields for $s\in D_X \cap [0,T)$
\begin{eqnarray*}
&& J_s^{+,1}\\&=&\int_\R \Bigl(F(X_{s+}+\sqrt{V(s+)-V^+(s)}y) -F(X_s)  \\ && \quad \quad -F'(X_s) (\Delta^+X_s +\sqrt{V(s+)-V^+(s)}y)\Bigr)\frac{e^{-y^2/2}}{\sqrt{2\pi}} dy\\
&=&  \int_\R  \int_0^1  (\Delta^+X_s +\sqrt{V(s+)-V^+(s)}y)^2  \\ &&  \times (1-u) F''((1-u)X_s+u(X_{s+}+\sqrt{V(s+)-V^+(s)}y)) du \frac{e^{-y^2/2}}{\sqrt{2\pi}}  dy.
\end{eqnarray*}
We now define $\epsilon:=((4a\lambda)^{-1}-1)/2$, where $a$ is the constant in the subexponential growth condition, let $\epsilon^*:=1/\epsilon$, and abbreviate
$$
l(s,y,u):=(1-u)X_s+u(X_{s+}+\sqrt{V(s+)-V^+(s)}y).
$$
Then, by Jensen's inequality, Fubini's theorem, and H\"older's inequality,

\begin{eqnarray*}
&& E[|J_s^{+,1}|^2]^{1/2}\\ &\leq&  \left(\int_\R  \int_0^1  E[(\Delta^+X_s +\sqrt{V(s+)-V^+(s)}y)^4\, |F''(l(s,y,u))|^2] du \frac{e^{-y^2/2}}{\sqrt{2\pi}}  dy\right)^{1/2} \\
&\leq & \left(\int_\R  \int_0^1  E[|F''(l(s,y,u))|^{2(1+\epsilon)}] du \frac{e^{-y^2/2}}{\sqrt{2\pi}}  dy\right)^{\frac{1}{2(1+\epsilon)}}  \\
&&\times  \left(\int_\R   E[|\Delta^+X_s +\sqrt{V(s+)-V^+(s)}y|^{4(1+\epsilon^*)}] \frac{e^{-y^2/2}}{\sqrt{2\pi}}  dy\right)^{\frac{1}{2(1+\epsilon^*)}} 
\end{eqnarray*}
The second factor is the square of the $L^{4(1+\epsilon^*)}$-norm of a centered Gaussian random variable with variance $E[(\Delta^+ X_s)^2]+V(s+)-V^+(s)$. Hence, there is a constant $d_{\epsilon^*}$ 
such that
\begin{eqnarray*}
&& \Bigl(\int_{\R} E\left[|\Delta^+X_s +\sqrt{V(s+)-V^+(s)}y|^{4(1+\epsilon^*)}\right]  \frac{e^{-y^2/2}}{\sqrt{2\pi}}dy\Bigr)^{\frac{1}{2(1+\epsilon^*)}} \\ &\leq& 
d_{\epsilon^*}(E[(\Delta^+ X_s)^2]+V(s+)-V^+(s)).
\end{eqnarray*}
The first factor is bounded by a constant $K_{\epsilon}$ independent of  $s$ by the subexponential growth condition. Indeed, by convexity,
\begin{eqnarray*}
&&  \int_\R  \int_0^1  E[|F''(l(s,y,u))|^{2(1+\epsilon)}] du \frac{e^{-y^2/2}}{\sqrt{2\pi}}  dy \\ &\leq & C^{2(1+\epsilon)} \int_\R  \max_{u\in \{0,1\}} E\left[e^{2(1+\epsilon)a|l(s,y,u)|^2}\right] \frac{e^{-y^2/2}}{\sqrt{2\pi}}  dy
\\ &\leq &  C^{2(1+\epsilon)} \Bigl(  E\left[e^{2(1+\epsilon)a|X_s|^2}\right] \\ &&+ \int_\R E\left[e^{2(1+\epsilon)a|X_{s+}+\sqrt{V(s+)-V^+(s)}y|^2}\right]\frac{e^{-y^2/2}}{\sqrt{2\pi}}  dy \Bigr)
\\ &\leq &  C^{2(1+\epsilon)} \sup_{t\in[0,T)}\int_\R  \left(e^{2(1+\epsilon)aV(t) z^2}+ e^{2(1+\epsilon)aV(t+) z^2}\right) \frac{e^{-z^2/2}}{\sqrt{2\pi}}  dz \\ &=:&K_\epsilon^{2(1+\epsilon)}<\infty.
\end{eqnarray*}
Thus,
\begin{equation} \label{eq:ito_jumps1}
 E[|J_s^{+,1}|^2]^{1/2} \leq K_{\epsilon} d_{\epsilon^*}(E[(\Delta^+ X_s)^2]+V(s+)-V^+(s)).
\end{equation}
Moreover, we clearly observe that 
\begin{equation} \label{eq:ito_jumps2}
 E[|J_s^{+,2}|^2]^{1/2} \leq \frac{1}{2} \left(\sup_{t\in [0,T]} E[|F''(X_t)|]^{1/2}\right)\, (E[(\Delta^+X_s)^2]  +V(s+)-V^+(s)).
\end{equation}
We can now deduce from \eqref{eq:ito_jumps1}, \eqref{eq:ito_jumps2},  and \eqref{eq:ito_ass} that the sum 
$
\sum_{s\in  D_X \cap [0,T)} J_{s}^+
$
converges absolutely  in $(L^2_X)$. In particular,
$$
S\left(\sum_{s\in  D_X \cap [0,T)}^\infty  J_{s}^+\right)(h)=\sum_{s\in  D_X \cap [0,T)} S( J_s^+)(h),\quad h\in \mathcal{A}_{W^*_2},
$$
Let us summarize the foregoing: The sum 
\begin{eqnarray}\label{eq:jumps_right1}
 && \sum_{s\in D_X \cap [0,T)} \Bigl(\psi_F(V(s+)-V^+(s), X_{s+})-F(X_s) -F'(X_s) \Delta^+X_s  \\
&& \quad\quad\quad\quad\quad - \frac{1}{2} F''(X_s)(E[(\Delta^+X_s)^2] +V(s+)-V^+(s))\Bigr) \nonumber
\end{eqnarray}
converges absolutely in $(L^2_X)$ and, due to \eqref{eq:Stransform_jumps}, its $S$-transform is given by 
\begin{eqnarray}\label{eq:jumps_right2}
  && \sum_{s\in D_X \cap [0,T)} \Bigl(S(\,\psi_F(V(s+)-V^+(s), X_{s+}))(h) -(S\,F(X_s))(h) \\ && \quad\quad-(S\,F'(X_s)) \Delta^+\inth(s) -
\frac{1}{2} (S\,F''(X_s)) \Delta^+V(s) \Bigr)  \nonumber
\end{eqnarray}
for $h\in \mathcal{A}_{W^*_2}$. The jumps from the left can be treated in the same way. The only difference is that we apply
$$
2E[X_s\Delta^-X_s]-E[(\Delta^-X_s)^2]  -(V(s-)-V^-(s))=V(s)-V(s-)=\Delta^-V(s)
$$
instead of \eqref{eq:hilf0003}, which explains the change of sign in front of the second derivative term compared to the corresponding term  resulting in the case of the jumps from the right. 
We finally obtain that
\begin{eqnarray}\label{eq:jumps_left1}
&&  \sum_{s\in D_X \cap (0,T]} \Bigl( F(X_s)- \psi_F(V(s-)-V^-(s), X_{s-})-F'(X_s) \Delta^-X_s \\
&& \quad\quad\quad\quad\quad + \frac{1}{2} F''(X_s)(E[(\Delta^-X_s)^2] +V(s-)-V^-(s))\Bigr)  \nonumber
\end{eqnarray}
converges absolutely in $(L^2_X)$ and its $S$-transform is given by 
\begin{eqnarray}\label{eq:jumps_left2}
 &&   \sum_{s\in D_X \cap (0,T]} \Bigl((S\,F(X_s))(h)- S(\,\psi_F(V(s-)-V^-(s), X_{s-}))(h) \\ && \quad\quad-(S\,F'(X_s)) \Delta^-\inth(s) -
\frac{1}{2} (S\,F''(X_s)) \Delta^-V(s) \Bigr) \nonumber
\end{eqnarray}
for $h\in \mathcal{A}_{W^*_2}$. 

We next discuss the integral with respect to the variance $V$. The subexponential growth condition \eqref{eq:growth} again ensures that 
$$
\int_0^T  E[|F''(X_s)|^2] d|V|(s)<\infty,
$$  
where $|V|$ denotes the total variation of $V$. Thus, by Fubini's theorem and H\"older's inequality, 
$$
\int_0^T  F''(X_s) dV(s) \in (L^2_X).
$$
Another application of Fubini's theorem then yields
\begin{equation}\label{eq:FV}
 S\left(\int_0^T  F''(X_s) dV(s) \right)(h)= \int_0^T  (S\,F''(X_s))(h)\, dV(s),
\end{equation}
for every $h\in \mathcal{A}_{W^*_2}$.

We can finally combine \eqref{eq:jumps_right1}--\eqref{eq:FV} with Proposition \ref{prop:ito} in order to show that the $S$-transform of 
\begin{eqnarray*}
 && F(X_T)-F(X_0)- \frac{1}{2} \int_0^T F''(X_s) dV(s) \\
& -&\sum_{s\in D_X \cap (0,T]}  \Bigl(F(X_s)- \psi_F(V(s-)-V^-(s), X_{s-})-F'(X_s) \Delta^-X_s \\
&& \quad\quad\quad\quad\quad + \frac{1}{2} F''(X_s)(E[(\Delta^-X_s)^2] +V(s-)-V^-(s))\Bigr) \\ 
 &-& \sum_{s\in D_X \cap [0,T)} \Bigl( \psi_F(V(s+)-V^+(s), X_{s+})-F(X_s) -F'(X_s) \Delta^+X_s \\
&& \quad\quad\quad\quad\quad - \frac{1}{2} F''(X_s)(E[(\Delta^+X_s)^2] +V(s+)-V^+(s))\Bigr)
\end{eqnarray*}
at every $h\in \mathcal{A}_{W^*_2}$ is given by the Henstock-Kurzweil integral
$$
\int_0^T (S\, F'(X_s))(h)\,d\inth(s).
$$
Hence, the Wick-Skorokhod integral  $\int_0^T F'(X_s) d^\diamond X_s$ exists and the asserted It\^o formula is valid.
\end{proof}

We close this section with  a simplified version of the It\^o formula in Theorem \ref{thm:Ito} as announced in \eqref{eq:Ito_rcll}. To this end, we assume that $X$ is \emph{stochastically RCLL}, i.e. 
for every $t\in [0,T)$ and every sequence $(t_n)$ converging to $t$ from the right, $X_{t_n}$ converges to $X_t$ in probability, and, moreover: For every 
$t\in (0,T]$ there is a random variable $X_{t-}\in (L^2_X)$ such that for every    sequence $(t_n)$ converging to $t$ from the left, $X_{t_n}$ converges to $X_{t-}$ in probability.
By Gaussianity, both limits also hold strongly in $(L^2_X)$. In particular, $X$ is weakly regulated with $X_{t+}=X_t$.
\begin{cor}\label{cor:ito}
 Suppose the centered Gaussian process $X$ satisfies the following assumptions: $X$ is stochastically RCLL, $\mathcal{A}_{W^*_2}$ is dense in $H_X$, the variance function $V$ 
of $X$ is of bounded variation, and
\begin{equation*}
\sum_{s\in D_X}  E[(\Delta^- X_s)^2]<\infty.
\end{equation*}
Assume $F\in \mathcal{C}^2(\mathbb{R})$ and $F,F',F''$ satisfy the growth condition \eqref{eq:growth} with $\lambda=\sup_{t\in [0,T]} V(t)$. 
Then, $\int_{0+}^T F'(X_{s-}) d^\diamond X_s$ exists and the following It\^o formula holds in $(L^2_X)$:
\begin{eqnarray*}
 F(X_T)&=&F(X_0)+\int_{0+}^T F'(X_{s-}) d^\diamond X_s + \frac{1}{2} \int_{0+}^T F''(X_{s-}) dV^c(s) \\
&& + \sum_{s\in D_X \cap (0,T]} \Bigl( F(X_s)- F(X_{s-})-F'(X_{s-}) \Delta^-X_s  \\ && \quad\quad+  F''(X_{s-})E[X_{s-}(\Delta^-X_s)]\Bigr). 
 \end{eqnarray*}
Here, $V^c$ denotes the continuous part of $V$, the set $D_X$ of stochastic discontinuities is at most countable and the sum converges absolutely in $(L^2_X)$.
\end{cor}
Note that $E[X_{s-}(\Delta^-X_s)]=0$, if $X$ is martingale. Hence, in view of Theorem \ref{thm:relation_ito}, Corollary \ref{cor:ito} contains the classical It\^o formula 
\eqref{eq:Ito_mart} for Gaussian martingales as special case. (Recall that the pathwise jumps of a Gaussian martingale occur only at the deterministic times of stochastic 
discontinuities).
\begin{proof}[Proof of Corollary \ref{cor:ito}]
 As $X$ is stochastically right-continuous, the rational span of $$\{X_t;\;t\in ([0,T)\cap \mathbb{Q})\cup \{T\}\}$$ is dense in $H_X$. Hence, condition (H) holds for $X$. Moreover, 
$V^{\pm }(t)=V(t\pm)$, because $X$ is stochastically RCLL, as already observed at the end of Example \ref{ex:F(X)}. Finally, $\Delta^+ X_{t}=0$ for every $t\in [0,T]$. 
Consequently, Theorem \ref{thm:Ito} is applicable and simplifies in the following way:
\begin{eqnarray*}
 && F(X_T)-F(X_0) \\ &=&\int_0^T F'(X_s) d^\diamond X_s + \frac{1}{2} \int_0^T F''(X_s) dV(s) \\
&& + \sum_{s\in D_X \cap (0,T]}  F(X_s)- F( X_{s-})-F'(X_s) \Delta^-X_s +\frac{1}{2} F''(X_s)E[(\Delta^-X_s)^2].
\end{eqnarray*}
As $V$ is  right-continuous, $\mu_V$ has no atom at $0$. Hence,
\begin{eqnarray*}
&& \frac{1}{2} \int_0^T F''(X_s) dV(s)=\frac{1}{2} \int_{0+}^T F''(X_s) dV(s)\\ &=& \frac{1}{2} \int_{0+}^T F''(X_{s-}) dV^c(s) + \frac{1}{2}\sum_{s\in  D_X \cap (0,T]} F''(X_s) \Delta^- V(s).
\end{eqnarray*}
Noting that $\frac{1}{2}(E[(\Delta^-X_s)^2]+\Delta^- V(s))=E[X_s\Delta^-X_s]$, we obtain 
\begin{eqnarray*}
 && F(X_T)-F(X_0) - \frac{1}{2} \int_{0+}^T F''(X_{s-}) dV^c(s) \\
&& - \sum_{s\in D_X \cap (0,T]} \Bigl( F(X_s)- F(X_{s-})-F'(X_{s-}) \Delta^-X_s \\ &&\quad\quad+  F''(X_{s-})E[X_{s-}(\Delta^-X_s)]\Bigr) \\ &=&
\int_0^T F'(X_s) d^\diamond X_s  -  \sum_{s\in D_X \cap (0,T]} \Bigl((F'(X_s)-F'(X_{s-})) \Delta^-X_s \\ && \quad \quad- (F''(X_s)E[X_{s}\Delta^-X_s]-F''(X_{s-})E[X_{s-}\Delta^-X_s])\Bigr),
\end{eqnarray*}
where, by the local Lipschitz continuity of $F'$ and the growth condition on $F''$, the sum on the right-hand side can be seen to converge absolutely in $(L^2_X)$. We now compute the 
$S$-transform of the right-hand side at $h\in \mathcal{A}_{W^*_2}$. Applying the right-continuity of $\inth$ and the analogue of \eqref{eq:Stransform_derivative}, we observe that it is given by 
\begin{eqnarray*}
 && \int_{0+}^T (S\, F'(X_s))(h)\,d\inth(s)-\sum_{s\in D_X \cap (0,T]}  S(F'(X_s)-F'(X_{s-}))(h) \Delta^- \inth(s) \\
&=& \int_{0+}^T (S\, F'(X_{s-}))(h)\,d\inth(s).
\end{eqnarray*}
Again, we may conclude that the Wick-Skorokhod integral $\int_{0+}^T F'(X_{s-}) d^\diamond X_s$ exists and that the asserted It\^o formula holds.
\end{proof}

\section{Subordinated fractional Brownian motion}\label{sec:fBm}

In this section, we consider processes of the form $X_t=B^H_{A_t}$ where $B^H$ is a fractional Brownian motion with Hurst parameter $H\in (0,1)$ and $A$ is a L\'evy subordinator independent of $B^H$, see \cite{Sa}, pp. 137--142, for background information on 
L\'evy subordinators. The resulting process 
$X$ has stationary increments and features jumps (unless the subordinator is deterministic, which we rule out for the remainder of the section).
If the subordinator $A$ is a Gamma process, then $B^H_{A_t}$ is known as \emph{fractional Laplace motion} and has been applied, e.g., as a model for hydraulic conductivity \citep{Mal} and sediment transport \citep{Gal}. Subordinated fractional 
Brownian motion has also been proposed as a risk process in actuarial mathematics, see \cite{DHJ}.

 Apparently, subordinated fractional Brownian motion $X_t=B^H_{A_t}$ itself is not a Gaussian process,
but $X$ becomes Gaussian after conditioning on the path of the subordinator. This fact can be used to extend our integration theory to subordinated fractional Brownian motion and to derive an It\^o formula for this class of processes. 
 Of course, this reasoning applies to a much wider class of conditionally Gaussian processes. 

Exploiting the independence of $B^H$ and $A$, the obvious generalization of the Wick-Skorokhod integral is the following:  
We first condition on the paths of the subordinator $A$, then integrate in the sense of Gaussian Wick-Skorokhod integration, and finally evaluate the integral along the paths of the subordinator.
This formally leads to the definition
$$
\int_0^T z(s,X) \, d^\diamond X_s = \left( \int_0^T z(s,X^a) \, d^\diamond X^a_s \right)_{|a=A}
$$
where $X^a_s:=B^H_{a(s)}$ for a deterministic nondecreasing function $a$ with right-continuous paths.

In order to define the integral rigorously, we denote by $\mathbb{D}([0,T])$ the space of RCLL real-valued functions on $[0,T]$ and equip it with the Borel-$\sigma$-field (generated by the Skorokhod topology). For every RCLL nondecreasing function $a:[0,T]\rightarrow \R$,
 we consider the Gaussian process $X^a_t=B^H_{a(t)}$, $t\in [0,T]$. Since,
\begin{equation}\label{eq:CVfBm}
Cov(X^a_t, X^a_s)=\frac{1}{2}(a(t)^{2H}+a(s)^{2H}-|a(t)-a(s)|^{2H} )
\end{equation}
 (cp. Example \ref{exmp:Skorokhod}), we observe that the covariance function is, for fixed $s$, a bounded variation function in $t$. In particular, each of the processes $X^a$ satisfies condition $(H)$. 
\begin{defi}
 Suppose that  $Z_t=z(t,X)$ for a measurable function $z:[0,T]\times \mathbb{D}([0,T])\rightarrow \R$ and $\mathbb{A}$ is a set of RCLL nondecreasing functions such that $P(\{A\in \mathbb{A}\})=1$. If $\int_0^T z(s, X^a) \,d^\diamond X^a_s$ exists for every $a\in \mathbb{A}$ and there is a measurable 
function $I_z:  \mathbb{D}([0,T]) \times \mathbb{D}([0,T])\rightarrow \R$ such that $\int_0^T z(s, X^a) \,d^\diamond X^a_s=I_z(X^a, a)$ for every $a\in \mathbb{A}$, then we define 
$$
\int_0^T Z_s \, d^\diamond X_s:= I_z(X,A).
$$
\end{defi}
A routine application of Fubini's theorem shows that the integral is (in case of existence) uniquely defined up to modification.

We can now state an It\^o formula for subordinated fractional Brownian motion in our framework. In the statement of the formula, we apply that any L\'evy subordinator can be represented 
as
$$
A_t=\sigma t+ \sum_{0<s\leq t} (\Delta A_s)
$$
for some constant $\sigma\geq 0$ and has only positive jumps, see \cite{Sa}, Theorem 21.5. We denote the L\'evy measure of $A$ by $\nu$. 
\begin{thm} \label{thm:subordinated}
 Suppose $H\geq 1/2$ or   $\int_{0+}^1 x^{2H} \nu(dx)<\infty$. If $F\in \mathcal{C}^2(\R)$ and $F$, $F'$, and $F''$ are bounded, then 
$\int_{0+}^T F'(X_{s-}) d^\diamond X_s$ exists and the following It\^o formula holds $P$-almost surely:
\begin{eqnarray*}
 F(X_T)&=&F(0)+\int_{0+}^T F'(X_{s-}) d^\diamond X_s + \sigma H \int_{0+}^T F''(X_{s-}) A_s^{2H-1} ds \\
&& + \sum_{s\in (0,T]} \Bigl( F(X_s)- F(X_{s-})-F'(X_{s-}) \Delta X_s  \\ && \quad\quad+ \frac{1}{2} F''(X_{s-})(A_s^{2H}-A_{s-}^{2H}-(\Delta A_s)^{2H})\Bigr). 
 \end{eqnarray*}
Here, the sum runs pathwise over the set of jump times of $A$ (which includes the jump times of $X$) and is absolutely convergent $P$-almost surely. Morever, $\sigma$ is the constant slope of the continuous part of $A$, and
the term $\sigma H \int_{0+}^T F''(X_{s-}) A_s^{2H-1} ds$ is to be read as zero, if $\sigma=0$.
\end{thm}

In the Brownian motion case $H=1/2$, the processes of the form $X_t=B^{1/2}_{A_t}$ are L\'evy processes and, in particular, include the class of CGMY-processes which were introduced in financial modeling 
in \cite{CGMY} and are among the most popular stock price models with jumps. As expected,
 our It\^o formula takes the same form as the classical one for L\'evy processes, if $H=1/2$. We again emphasize that the covariance function of fractional Brownian motion is visible in the additional 
 second derivative term of the jump component. 

\begin{proof}
 We first define $\mathbb{A}$ as the set of nondecreasing RCLL functions on $[0,T]$ of the form $a(t)=\sigma t+ \sum_{0<s\leq t} (\Delta a(s))$ such that $\sum_{0<s\leq T} (\Delta a(s))^{2H}<\infty$. We claim that $P(\{A\in  \mathbb{A} \})=1$. To this end, we
decompose
\begin{equation}\label{eq:hilfr01}
\sum_{0<s\leq T} (\Delta A_s)^{2H} = \sum_{0<s\leq T} (\Delta A_s)^{2H}{\bf 1}_{\{0<\Delta A_s\leq 1\}}  +\sum_{0<s\leq T} (\Delta A_s)^{2H}{\bf 1}_{\{\Delta A_s>1\}}.
\end{equation}
As $A_t$ has only finitely many jumps of size larger than one, the second term is finite $P$-almost surely. The first term is nonnegative and satisfies 
$$
E\left[\sum_{0<s\leq T} (\Delta A_s)^{2H}{\bf 1}_{\{0<\Delta A_s\leq 1\}} \right]= T\int_{0+}^1 x^{2H} \nu(dx)
$$
by the definition of the L\'evy measure. The right-hand side is finite for $H<1/2$ by assumption and for $H\geq1/2$, because any L\'evy subordinator fulfills 
$$
\int_{0+}^1 x \nu(dx)< \infty,
$$
see \cite{Sa}, Theorem 21.5. Hence, the first term in \eqref{eq:hilfr01} is also finite $P$-almost surely. 

We next consider the set $\mathbb{A}'$ of RCLL functions such that $\sum_{0<s\leq T} (\Delta x(s))^{2}<\infty$, and define the functional
\begin{eqnarray*}
 I(x,a)&:=&{\bf 1}_{\{a\in \mathbb{A}\}\cap\{x\in \mathbb{A}' \}}\Biggl( 
F(x(T))-F(0)- \sigma H \int_{0+}^T F''(x({s-})) a(s)^{2H-1} ds \\
&& - \sum_{s\in (0,T]} \Bigl( F(x(s))- F(x({s-}))-F'(x({s-})) \Delta x(s)  \\ && \quad\quad+ \frac{1}{2} F''(x({s-}))(a(s)^{2H}-a({s-})^{2H}-(\Delta a(s))^{2H})\Bigr)\Biggr).
\end{eqnarray*}
Note that the assumptions on $F$ ensure that the sum over the jumps and the integral (if $\sigma>0$) converge. We emphasize that $P(\{X^a\in \mathbb{A}'\})=1$ for every $a\in \mathbb{A}$, and 
 $P(\{X\in \mathbb{A}'\})=1$. The first identity holds, because
$$
E\left[\sum_{s\in (0,T]} (\Delta X^a_s)^2 \right]=  \sum_{s;\; \Delta a(s)\neq 0} E[|B^H_{a(s)}-B^H_{a(s-)}|^2]= \sum_{s\in (0,T]} (\Delta a(s))^{2H}<\infty.
$$
The second identity is a consequence of Fubini's theorem and the first one, since
$$
P(\{X\in \mathbb{A}'\})=\int_{\mathbb{A}} P(\{X^a\in \mathbb{A}'\}) P_A(da) =1.
$$
It, thus, remains to show that, for every $a\in \mathbb{A}$,
\begin{eqnarray*}
\int_0^T F'(X^a_{s-})\,d^\diamond X^a_s&=& 
F(X^a_{T})-F(0)- \sigma H \int_{0+}^T F''(X^a_{s-}) a(s)^{2H-1} ds \\
&& - \sum_{s\in (0,T]} \Bigl( F(X^a_{s})- F(X^a_{s-})-F'(X^a_{s-}) \Delta X^a_s  \\ && \quad + \frac{1}{2} F''(X^a_{s-})(a(s)^{2H}-a({s-})^{2H}-(\Delta a(s))^{2H})\Bigr).
\end{eqnarray*}
This equation can easily be reduced to Corollary \ref{cor:ito}. To this end we note, that by continuity of $B^H$, $D_{X^a}=\{s;\; \Delta a(s)\neq 0\}\supset \{s;\; \Delta X^a_s \neq 0\}$. Hence, the (pathwise) sum 
over $s \in (0,T]$ can be replaced by the sum over $s\in D_{X^a}\cap (0,T]$. Moreover, by \eqref{eq:CVfBm},
$$
E[X^a_{s-}(\Delta^-X^a_s)]= E[X^a_{s-}(\Delta X^a_s)]=\frac{1}{2} (a(s)^{2H}-a({s-})^{2H}-(\Delta a(s))^{2H}),
$$
and so the sum over the jumps coincides with the corresponding sum in Corollary \ref{cor:ito}. We finally note that the variance of $X^a$ is given by $V_a(t)=a(t)^{2H}$. Hence, by the chain rule in Theorem 
\ref{thm:chain_rule}, the continuous part of $V_a$ is zero, if $\sigma=0$, and otherwise equals
$$
V^c_a(t)=2H \sigma  \int_{0}^t a(s)^{2H-1}  ds.
$$
Hence, 
$$
\sigma H \int_{0+}^T F''(X^a_{s-}) a(s)^{2H-1} ds=\frac{1}{2} \int_{0+}^T F''(X^a_{s-})  dV^c_a(s),
$$
which finishes the reduction to Corollary \ref{cor:ito} and the proof.
\end{proof}

\begin{exmp}
 (i) One of the best studied subordinated fractional Brownian motion is fractional Laplace motion, where the subordinator is a Gamma process $A$, i.e. a pure jump L\'evy process whose L\'evy measure has the density 
$$
\rho(x)={\bf 1}_{(0,\infty]}(x) c x^{-1}e^{-\lambda x}
$$
with respect to the Lebesgue measure, for positive constants $c$ and $\lambda$. The integrability condition on the L\'evy measure imposed in Theorem \ref{thm:subordinated} 
holds for every $H\in (0,1)$, and so our It\^o formula applies to the full range of Hurst parameters. Fractional Laplace motion 
has heavier tails than a normal distribution and features a stochastic self-similarity property. It has long-range dependence if and only if $H>1/2$, and its distribution is infinite divisible, if and only if $H\geq 1/2$. For a proof of these properties 
we refer to \cite{KMP}. \\[0.1cm]
(ii) If $A$ is an inverse Gaussian subordinator, i.e. a pure jump L\'evy process whose L\'evy measure has the density
$$
\rho(x)={\bf 1}_{(0,\infty]}(x) c x^{-3/2}e^{-\lambda x}
$$
with respect to the Lebesgue measure ($c,\lambda>0$), then $X_t=B^H_{A_{t}}$ is known as \emph{fractional inverse Gaussian process}.
It  can be obtained as scaling limit of  continuous time random walks with correlated jumps 
\citep{KMV}, has long-range dependence for $H>1/2$ \citep{KV}, and fails to have an infinite divisible distribution for $H<1/2$ \citep{Wal}. The integrability condition on the L\'evy measure imposed in Theorem \ref{thm:subordinated} 
holds for $H>1/4$, and so our It\^o formula applies to this range of Hurst parameters.
\end{exmp}

\section{Regularity in the Cameron-Martin space}\label{sec_reg}

In this section, we study the regularity of the elements of the Cameron-Martin space of $X$ as required in condition (H). We provide sufficient conditions in terms of the mixed planar variation of the covariance function,
the quadratic variation of $X$, and finally consider the case, when $X$ has an integral representation with respect to a Gaussian martingale.

The first criterion is formulated in terms of the mixed planar variation and is taken from \citet{Fal}.
\begin{thm}\label{thm:PV}
 Suppose that the covariance function of $X$ is of finite $(1,p)$-planar variation for some $p\geq 1$ in the following sense: There is a constant $K>0$ such that for every 
partition $\pi=\{0=t_0<t_1<\cdots<t_n=T\}$
$$
\sum_{j=1}^n \left(\sum_{i=1}^n  |E[(X_{t_i}-X_{t_{i-1}})(X_{t_j}-X_{t_{j-1}})]|\right)^p\leq K.
$$
Then, $X$ is $W^*_2$-regular.
\end{thm}
\begin{proof}
 By the proof of Theorem 1.1 in \cite{Fal}, every element of the Cameron-Martin space of $X$ then belongs to $W_q$ for $q=2p/(p+1)<2$. (Note that continuity of the paths of $X$ is assumed in 
\cite{Fal}, but the proof of their Theorem 1.1 does not make use of this property). Hence, by Remark \ref{rem:Wp} (i), $X$ is $W^*_2$-regular.
\end{proof}
Theorem 2.2 in \citet{Fal} provides sufficient conditions for the planar $(1,p)$-planar variation to be finite, if $X$ is stochastically continuous. See also Examples 2.4--2.16 in the same reference. 
In particular, we may derive from their Example 2.12 that bifractional Brownian motion is $W^*_2$-regular for every choice of the parameters $H\in (0,1)$, $K\in (0,1]$.

The next theorem states that an RCLL centered Gaussian process is $W^*_2$-regular, if the continuous part of the quadratic variation is deterministic and it has fixed discontinuities only. 
Here is the precise statement. 
\begin{thm}\label{thm:QV}
 Suppose $X$ has RCLL paths and that  jumps of the paths of $X$ only take place at the times of stochastic discontinuities $s\in D_X$. 
If 
$$
\sum_{s\in D_X\cap (0,T]} E[(\Delta X_s)^2] <\infty
$$
 and there is a deterministic function $v:[0,T]\rightarrow \R$ such that for every $t\in [0,T]$
$$
\sum_{i=1}^n (X_{t_i\wedge t}-X_{t_{i-1}\wedge t})^2 \rightarrow v(t) + \sum_{s\in (0,t] } (\Delta X_s)^2
$$
in probability as the mesh size of the partition $\pi=\{0=t_0<t_1<\cdots<t_n=T\}$ tends to zero, then $X$ is $W^*_2$-regular.
\end{thm}
\begin{proof}
 The assumptions guarantee that
$$
\sum_{s\in (0,t] } (\Delta X_s)^2=\sum_{s\in (0,t]\cap D_X } (\Delta X_s)^2
$$
and that this sum converges absolutely in $(L^2_X)$, because
$$
\sum_{s\in (0,t]\cap D_X } E[|(\Delta X_s)^2|^2]^{1/2} =\sqrt{3} \sum_{s\in (0,t]\cap D_X } E[(\Delta X_s)^2]<\infty.
$$ 
As $|\Delta \inth(s)|^2=E[(\exp^\diamond(h)-1)(\Delta X_s)^2]$ due to \eqref{eq:shift}, we, thus, obtain, for every $\inth \in CM_X$,
$$
 \sigma_2(\inth)=\sum_{s\in (0,T] } |\Delta \inth(s)|^2= \sum_{s\in (0,T]\cap D_X} |\Delta \inth(s)|^2<\infty.
$$
Let $(\pi_k)$ be a sequence of partitions of $[0,T]$, whose mesh size tends to zero. Then, by Theorem 3.50 in \cite{Ja}
\begin{equation*}\label{eq:hilf0006}
\sum_{t_i\in \pi_k\setminus \{0\}} (X_{t_i}-X_{t_{i-1}})^2 -  \sum_{s\in D_X\cap (0,T] } (\Delta X_s)^2 \rightarrow v(T) 
\end{equation*}
in $(L^2_X)$ as $k$ goes to infinity, which in turn implies,
\begin{equation*}
 \sum_{t_i\in \pi_k\setminus \{0\}} E[(\exp^\diamond(h)-1) (X_{t_i}-X_{t_{i-1}})^2] -  \sum_{s\in D_X\cap (0,T] } E[(\exp^\diamond(h)-1)(\Delta X_s)^2] \rightarrow 0.
\end{equation*}
By \eqref{eq:shift}, the left-hand side equals,
$$
\sum_{t_i\in \pi_k\setminus \{0\}} (\inth(t_i)-\inth(t_{i-1}))^2 -  \sum_{s\in D_X\cap (0,T] }  |\Delta \inth(s)|^2 
$$ 
 Thus,
$$
\lim_{k\rightarrow \infty} \sum_{t_i\in \pi_k\setminus \{0\}} (\inth(t_i)-\inth(t_{i-1}))^2= \sum_{s\in (0,T]} |\Delta \inth(s)|^2 = \sigma_2(\inth).
$$
 This convergence along the mesh size clearly implies convergence along the direction of refinement of partitions, as required 
in the definition of $W^*_2$.
\end{proof}

We finally assume that $(M_t)_{t\in \R}$ is a a two-sided  RCLL Gaussian martingale with variance function $V_M$ and $X$ has an integral representation of the form
\begin{equation}\label{eq:integral_representation}
X_t=\int_{-\infty}^{+\infty} \g_t(s) dM_s,
\end{equation}
where $\g_t\in L^2(\R, dV_M)$ for every $t\in[0,T]$ and the integral is understood in the It\^o sense. 
In particular, $H_X$ is then a closed subspace of $H_M=\{\int \f(s) dM_s;\; \f\in L^2(\R, dV_M)\}$ and we denote by $\pi_{H_X}$ the orthogonal projection on $H_X$. We also extend $\g_t$ to $t\in\R$ by setting 
$\g_t:=\g_T$ for $t>T$ and $\g_t:=\g_0$ for $t<0$. 
\begin{prop}\label{prop:integral_representation}
 Suppose $\mathfrak{A}$ is a dense subset of $L^2(\R, dV_M)$. Then, the  functions of the form
$$
[0,T]\rightarrow \R,\; t\mapsto E\left[X_t \int_\R \f(s) dM_s\right],\quad f\in \mathfrak{A} 
$$
constitute a dense subset of the Cameron-Martin space $CM_X$ of $X$.
\end{prop}
\begin{proof}
 For $\f \in \mathfrak{A}$ let $f:= \pi_{H_X}(\int \f(s) dM_s)$. Then, $E[X_\cdot \,\int \f(s) dM_s]=E[X_\cdot \,f]$ is indeed an element of the Cameron-Martin space of $X$. 
However, $\{\int \f(s) dM_s;\; \f\in \mathfrak{A}\}$ is dense in $H_M$ by the It\^o isometry, which  implies that $\{\pi_{H_X}(\int \f(s) dM_s);\; \f\in \mathfrak{A}\}$ is dense in $H_X$. This completes the proof.
\end{proof}

The following theorem provides a simple sufficient condition for $X$ to be $W^*_2$-dense in terms of the integral kernel $\g_t$ in the representation \eqref{eq:integral_representation} and the adjoint 
of the densely defined linear operator $$K: L^2(\R, dV_M)\rightarrow L^2(\R, dV_M),\; {\bf 1}_{(a,b]} \mapsto \g_b-\g_a.$$

\begin{thm}\label{thm:integral_representation}
$X$ is $W^*_2$-dense, if the adjoint operator $K^*$ of $K$ is densely defined, i.e.:
The set $D(K^*)$, consisting of those 
$\f \in L^2(\R,dV_M)$ such that there is an $K^*\f \in L^2(\R,dV_M)$ satisfying
$$
\int_\R (\g_t(s)-\g_0(s)) \f(s) dV_M(s) =\int_0^t K^*\f(s) dV_M(s)
$$
for every $t\in [0,T]$, is dense in $L^2(\R,dV_M)$.
\end{thm}
Note that by Theorem VII.2.3 in \cite{Yosida} the sufficient condition in Theorem \ref{thm:integral_representation} is equivalent to the closability of the operator $K$.
\begin{proof}
By the previous proposition, the functions of the form 
$$
t\mapsto \int_0^t K^*\f(s) dV_M(s)+ \int_\R \g_0(s) \f(s) dV_M(s)
$$
are then dense in $CM_X$, because by It\^o's isometry,
$$
E\left[X_t \int_\R \f(s) dM_s\right]= \int_\R \g_t(s) \f(s) dV_M(s).
$$
These functions are clearly of bounded variation, and, hence, members of $W^*_2$.
\end{proof}

\section{Comparison to the literature} \label{sec_comparison}

We finally explain how condition (H) can  be verified in the literature on It\^o's formula for Gaussian processes in the Skorokhod sense. In all the cases discussed below, the authors 
assume or show that the variance function $V$ of $X$ is of bounded variation and continuous:
\begin{enumerate}
 \item \citet{AMN}: The authors assume an integral representation of the form \eqref{eq:integral_representation} with respect to a Brownian motion $M=W$ (on the interval $[0,T]$). 
In the more general `singular' case \citep[Theorem 1]{AMN}, condition (K3) entails that  $t\mapsto E[X_t \int_0^T \f(s) dW_s]$ is continuous and of bounded variation for every step function $\f$.
Hence, by Proposition \ref{prop:integral_representation}, $X$ is $W^*_2$-dense and, moreover, $X$ is stochastically continuous \citep[Theorem 8.21]{Ja}. In particular,  condition (H)
is satisfied.
\item \citet{MV}: Again an integral representation of the form \eqref{eq:integral_representation} with respect to a Brownian motion $M=W$ on an interval is supposed. Proposition 15 in \cite{MV} shows 
that the assumptions of Theorem \ref{thm:integral_representation} are satisfied. Hence, a dense subset of the Cameron-Martin space of $X$ 
is absolutely continuous with respect to the Lebesgue measure with square integrable density. This again implies that (H) holds and $X$ is stochastically continuous. 
\item \citet{NT}: In this reference $X$ is supposed to have continuous paths. The key assumption   is  that $X$ has zero planar quadratic variation along sufficiently uniform partitions.
It\^o's formula can be recovered by our techniques under their assumptions, if the notions such as $W^*_2$-regularity are also formulated in terms of these sufficiently uniform partitions only. However,
in our general framework of Gaussian processes with stochastic discontinuities it does not seem to be possible to restrict to sufficiently uniform partitions as suggested in  
\cite{NT} in the continuous case. 
\item \citet{KRV}: The authors assume that $X$ is stochastically continuous and has a covariance measure. By their Remark 3.1 the latter property is equivalent to the property that 
the covariance function of $X$ has finite planar (1,1)-variation. By Theorem \ref{thm:PV}, $X$ is $W^*_2$-regular and, in particular, satisfies (H).
\item \citet{NT2}: Here the authors assume that $X$ has continuous paths and satisfies the quadratic variation property in our Theorem \ref{thm:QV}. Hence, $X$ is stochastically continuous,
 $W^*_2$-regular, and
 satisfies (H). 
\item \citet{LN} and \citet{HJT}: The assumptions in these two references include that, for every $s\in[0,T]$, the covariance function $t\mapsto R(t,s)$ of $X$ is absolutely continuous with respect to the Lebesgue measure. Thus, 
$X$ is stochastically continuous and satisfies (H) by Remark \ref{rem:Wp} (ii).
\item \citet{AK}: The authors consider a class of Gaussian stationary increment processes and show in their equation (5.1) that the $S$-transform 
$t\mapsto (SX_t)(h)$ of $X$ is absolutely continuous with respect to the Lebesgue measure for a dense subspace of $H_X$. Hence, by \eqref{eq:Stransform1}, $X$ is $W^*_2$-dense and satisfies (H), because it also 
is stochastically continuous. 
\item \citet{Leb}: $X$ is defined in terms of an integral representation of the form \eqref{eq:integral_representation} with respect to a two-sided Brownian motion $M=W$. 
The conditions on the kernel $\g_t$ are such that $t\mapsto E[X_t \int_\R \f dW]$ is absolutely continuous with respect to the Lebesgue measure for every Schwartz function $\f$, see Remark 
1 in \cite{Leb}.  As the Schwartz functions are dense in $L^2(\R,dt)$,   Proposition \ref{prop:integral_representation} again implies that $X$ is stochastically 
continuous and that (H) is satisfied.   
\end{enumerate}

We emphasize again that in all these references, the Gaussian process $X$ is stochastically continuous and, hence, It\^o's formula has the form \eqref{eq:Ito_cts} without jump terms, while 
the main contribution of the present paper is to understand the influence of the stochastic discontinuities on the Gaussian It\^o formula.

\appendix

\section{A chain rule for the Henstock-Kurzweil integral}

In this appendix, we prove a chain rule for the Henstock-Kurzweil integral. The general lines of proof closely follow the arguments in \citet{No}. We have, however, to deal 
with a case of `mixed' regularity which is not covered there. Here is what we are going to show.
\begin{thm}\label{thm:chain_rule}
Suppose $u_1 \in W_2^*([0,T]),\; u_2 \in W_1([0,T])$ (i.e, of bounded variation), and let $u:=(u_1,u_2)$. Denote $$S_i:=\left[\inf_{t\in [0,T]} u_i(t), \sup_{t\in [0,T]} u_i(t)\right],\quad  i=1,2.$$ Suppose $G\in \mathcal{C}^1(S_1\times S_2;\R)$ 
and such that there is a constant $K_1 \geq 0$ and a continuous function $K:S_1\times S_2 \times S_2 \rightarrow \R_{\geq 0}$ satisfying $K(x_1,x_2,x_2)=0$ and  
\begin{eqnarray}\label{eq:reg_chain_rule}
 \left|\frac{\partial G}{\partial x_1}(x_1,x_2)-  \frac{\partial G}{\partial x_1}(y_1,x_2)\right|&\leq& K_1|x_1-y_1|  \\
\left|\frac{\partial G}{\partial x_1}(x_1,x_2)-  \frac{\partial G}{\partial x_1}(x_1,y_2)\right|&\leq& K(x_1,x_2,y_2) \; |x_2-y_2|^{1/2}\nonumber
\end{eqnarray}
for every $x_1,y_1\in S_1$ and $x_2,y_2\in S_2$. Then, $\int_0^T \frac{\partial G}{\partial x_1}(u(s))\, du_1(s)$ exists as Henstock-Kurzweil integral and
\begin{eqnarray*}
&& G(u(T))-G(u(0))\\ &=& \int_0^T \frac{\partial G}{\partial x_1}(u(s))\, du_1(s)+ \int_0^T \frac{\partial G}{\partial x_2}(u(s))\, du_2(s) \\ &&+
 \sum_{s\in (0,T]} G(u(s))-G(u(s-))-  \frac{\partial G}{\partial x_1}(u(s))\Delta^-u_1(s)- \frac{\partial G}{\partial x_2}(u(s))\Delta^-u_2(s) \\ 
&&+
 \sum_{s\in [0,T)} G(u(s+))-G(u(s))-  \frac{\partial G}{\partial x_1}(u(s))\Delta^+u_1(s)- \frac{\partial G}{\partial x_2}(u(s))\Delta^+u_2(s).
\end{eqnarray*}
Here, both sums converge absoulutely. 
\end{thm}
Note first that the integral with respect to $u_2$ exists as Lebesgue-Stieltjes integral, because $u_2$ is of bounded variation and the integrand is regulated. The chain rule above, thus, holds 
as a consequence of Theorem 4.2 (with $\alpha=1$) in \citet{No} under the stronger Lipschitz condition 
\begin{eqnarray*}
 && \left|\frac{\partial G}{\partial x_1}(x_1,x_2)-  \frac{\partial G}{\partial x_1}(y_1,y_2)\right|+\left|\frac{\partial G}{\partial x_2}(x_1,x_2)-  
\frac{\partial G}{\partial x_2}(y_1,y_2)\right|\\&\leq& K_1(|x_1-y_1|+|x_2-y_2|),
\end{eqnarray*}
which is not satisfied in our application of this chain rule in Section \ref{sec_ito}. 

As in \citet{No}, we shall show that the chain rule is valid in the sense of Young-Stieltjes integration and then
 exploit the relationship between the Young-Stieltjes integral and the Henstock-Kurzweil 
integral as stated in  \citet[][Part I, Theorem F.2]{DN}. We, thus, first recall the construction of the Young-Stieltjes integral. 

A \emph{Young-tagged partition} $\tau:=\{((s_{i-1},s_i),y_i);\;i=1,\ldots,n\}$ of the interval $[0,T]$, by definition, satisfies $0=s_0<s_1<\ldots s_n=T$ and $y_i\in(s_{i-1},s_i)$, whereas 
in a tagged partition the tag point $y_i$ lies in the closed interval $[s_{i-1},s_i]$. Given such a Young-tagged partition $\tau$ and regulated functions $u,r:[0,T]\rightarrow \R$, one considers 
the \emph{Young-Stieltjes sums}
$$
S_{YS}(u,r,\tau):=\sum_{i=1}^n u(s_{i-1})\Delta^+r(s_{i-1})+u(y_i)(r(s_{i}-)-r(s_{i-1}+))+ u(s_{i})\Delta^-r(s_{i})
$$
The \emph{Young-Stieltjes integral} of $u$ with respect to $r$ is said to exist, if there is a real number $I$ such that, for every $\epsilon>0$, there is a 
Young-tagged partition $\chi$ such that for every refinement 
$\tau$ of $\chi$
$$
|S_{YS}(u,r,\tau)-I|<\epsilon.
$$
In this case, $I$ is defined to be the value of this integral. In the above, a Young-tagged partition $\tau=\{((s_{i-1},s_i),y_i);\;i=1,\ldots,n\}$ 
is said to be a \emph{refinement} of a Young-tagged partition $\chi$, if the partition $\kappa(\tau):=\{s_0,\ldots, s_n\}$ is a refinement of the partition $\kappa(\chi)$, i.e. $\kappa(\tau)\supset \kappa(\chi)$.

\begin{proof}[Proof of Theorem \ref{thm:chain_rule}]
 Define $S:=S_1\times S_2$, $S':=S_1\times S_2\times S_2$.  We consider for every Young-tagged partition $\tau:=\{((s_{i-1},s_i),y_i);\;i=1,\ldots,n\}$
\begin{eqnarray*}
V^+(\tau)&=& \sum_{i=0}^{n-1} G(u(s_i+))-G(u(s_i))-  \frac{\partial G}{\partial x_1}(u(s_i))\Delta^+u_1(s_i) \\ && \quad\quad- \frac{\partial G}{\partial x_2}(u(s_i))\Delta^+u_2(s_i) \\
V^-(\tau)&=& \sum_{i=1}^n G(u(s_i))-G(u(s_i-))-  \frac{\partial G}{\partial x_1}(u(s_i))\Delta^-u_1(s_i)\\ &&\quad\quad- \frac{\partial G}{\partial x_2}(u(s_i))\Delta^-u_2(s_i) \\
R(\tau)&=& \sum_{i=1}^n G(u(s_i-))-G(u(s_{i-1}+))- \frac{\partial G}{\partial x_1}(u(y_i))\\ &&\quad\times (u_1(s_i-)-u_1(s_{i-1}+))\ - \frac{\partial G}{\partial x_2}(u(y_i))(u_2(s_i-)-u_2(s_{i-1}+)).
\end{eqnarray*}
Then,
\begin{eqnarray}\label{eq:YS_decomposition}
 S_{YS}\left(\frac{\partial G}{\partial x_1}(u),u_1,\tau\right)&=&G(u(T))-G(u(0))-S_{YS}\left(\frac{\partial G}{\partial x_2}(u),u_2,\tau\right) \\ && -V^+(\tau)-V^-(\tau)-R(\tau).\nonumber
\end{eqnarray}
We fix an arbitrary $\epsilon>0$. 
As $\frac{\partial G}{\partial x_2}(u)$ is regulated and $u_2$ is of bounded variation, the Young-Stieltjes integral of $\frac{\partial G}{\partial x_2}\circ u$ with respect to $u_2$ exists 
and coincides with the Henstock-Kurzweil integral and the Lebesgue-Stieltjes integral by Theorems I.4.2, I.F.2 and Corollary II.3.20 in \citet{DN}. Hence, there is a Young-tagged partition $\chi$ of $[0,T]$ such that for all refinements $\tau$ of $\chi$
\begin{equation}
 \left|S_{YS}\left(\frac{\partial G}{\partial x_2}(u),u_2,\tau\right)-\int_0^T \frac{\partial G}{\partial x_2}(u(s))\,du_2(s) \right|<\epsilon/4.
\end{equation}
We next treat the jumps from the right. Suppose $s\in[0,T)$ such that $\Delta^+u_l(s) \neq 0$ for $l=1$ or $l=2$. Note that the set of such time points $s$ is at most countable, since $u_1$ and $u_2$ are regulated. By the 
mean value theorem there are $\theta_l\in [u_l(s)\wedge u_l(s+), u_l(s)\vee u_l(s+)]$, $l=1,2$ such that for $\theta=(\theta_1,\theta_2)$
$$
G(u(s+))-G(u(s))= \frac{\partial G}{\partial x_1}(\theta)\Delta^+u_1(s)+ \frac{\partial G}{\partial x_2}(\theta)\Delta^+u_2(s).
$$
Define
$$
V^+_s:=G(u(s+))-G(u(s))-  \frac{\partial G}{\partial x_1}(u(s))\Delta^+u_1(s)- \frac{\partial G}{\partial x_2}(u(s))\Delta^+u_2(s).
$$
Let $K_2:=\max_{(x_1,x_2,\tilde x_2)\in S'} K(x_1, x_2, \tilde x_2)$ and $K_3:= 2\max_{(x_1,x_2)\in S} |\frac{\partial G}{\partial x_2}(x_1,x_2)|$.
Then, by \eqref{eq:reg_chain_rule} and  Young's inequality,
\begin{eqnarray*}
 |V_s^+|&\leq& |\Delta^+u_1(s)|(K_1 |u_1(s)-\theta_1|+K_2|u_2(s)-\theta_2|^{1/2}) + K_3|\Delta^+u_2(s)| \\
&\leq & (K_1+K_2/2) |\Delta^+u_1(s)|^2+ (K_3+K_2/2)|\Delta^+u_2(s)|.
\end{eqnarray*}
Hence, for some constant $\tilde K>0$
$$
\sum_{s\in [0,T)} |V_s^+| \leq  \tilde K\left(\sum_{s\in [0,T)}  |\Delta^+u_1(s)|^2+\sum_{s\in [0,T)}  |\Delta^+u_2(s)| \right)<\infty,
$$
because $u_1\in W^*_2$ and $u_2$ is of bounded variation. Thus, the sum over the jumps from the right in the asserted chain rule converges absolutely. We can, then, find a finite 
subset $\mu\subset [0,T)$ such that 
$$
\sum_{s\in [0,T)\setminus \mu} |V_s^+| \leq  \epsilon/4.
$$
By passing to a refinement, if necessary, we can assume without loss of generality, that $\mu\subset \kappa(\chi)$. Then, for every refinement $\tau$ of $\chi$
\begin{eqnarray}
&& \Bigl| V^+(\tau)- \sum_{s\in [0,T)} G(u(s+))-G(u(s))-  \frac{\partial G}{\partial x_1}(u(s))\Delta^+u_1(s) \nonumber\\ &&\quad\quad- \frac{\partial G}{\partial x_2}(u(s))\Delta^+u_2(s)\Bigr|\leq \epsilon/4.
\end{eqnarray}
By the same argument, the sum over the jumps from the left converges absolutely and for every refinement $\tau$ of $\chi$
\begin{eqnarray}
&&\Bigl| V^-(\tau)- \sum_{s\in (0,T]} G(u(s))-G(u(s-))-  \frac{\partial G}{\partial x_1}(u(s))\Delta^-u_1(s)\nonumber\\ &&\quad\quad - \frac{\partial G}{\partial x_2}(u(s))\Delta^-u_2(s)\Bigr|\leq \epsilon/4.
\end{eqnarray}
It remains to treat the remainder term $R(\tau)$. Again, by the mean value theorem, there are $\theta_{l,i}\in [u_l(s_i-)\wedge u_l(s_{i-1}+), u_l(s_i-)\vee u_l(s_{i-1}+)]$, $l=1,2$, 
$i=1,\ldots, n$, such that for $\theta_i:=(\theta_{1,i},\theta_{2,i})$, $i=1,\ldots,n$,
\begin{eqnarray*}
 R(\tau)&=& \sum_{i=1}^n \left(\frac{\partial G}{\partial x_1}(\theta_i)- \frac{\partial G}{\partial x_1}(u(y_i))\right)(u_1(s_i-)-u_1(s_{i-1}+)) \nonumber \\ && +
\sum_{i=1}^n \left(\frac{\partial G}{\partial x_2}(\theta_i) -\frac{\partial G}{\partial x_2}(u(y_i))\right)(u_2(s_i-)-u_2(s_{i-1}+))\\&=:&(I)+(II).
\end{eqnarray*}
In order to estimate these two terms separately, we need some extra notation. We write
$$
\Osc(u, E):=\max_{l=1,2}\; \sup\{|u_l(s)-u_l(t)|;\; s,t \in E)\},\quad E\subset [0,T],
$$
for the oscillation of $u$ over the set $E$. We denote the total variation of $u_2$ over $[0,T]$ by $v_1(u_2)$ and the 2-variation of $u_1$ over the open interval 
$(s_{i-1},s_{i})$ by $v_2(u_1,(s_{i-1},s_i))$, i.e.
\begin{eqnarray*}
&&v_2(u_1,(s_{i-1},s_i))\\&:=&\sup\Bigl\{ \sum_{j=1}^m |u_1(t_j)-u_1(t_{j-1})|^2;\; m\in \N,\; s_{i-1}<t_0<t_1<\cdots<t_m<s_{i}\Bigr\}.
\end{eqnarray*}
Noting that, for every $l=1,2$ and $i=1,\ldots,n$ and $p\geq 1$, 
$$
|u_l(y_i)-\theta_{l,i}|^p\leq \max\{ |u_l(s_{i}-)-u_l(y_i)|^p, |u_l(y_i)- u_l(s_{i-1}+)|^p\},
$$
we obtain, by \eqref{eq:reg_chain_rule} and Young's inequality,
\begin{eqnarray*}
 |(I)|&\leq& \sum_{i=1}^n K_1 |\theta_{1,i}-u_1(y_i)||u_1(s_i-)-u_1(s_{i-1}+)|  \\ && +\sum_{i=1}^n K(\theta_{1,i},\theta_{2,i}, u_2(y_i))|\theta_{2,i}-u_2(y_i)|^{1/2}  |u_1(s_i-)-u_1(s_{i-1}+)| 
\\ &\leq & \sum_{i=1}^n \frac{K_1+K_2}{2}   |u_1(s_i-)-u_1(s_{i-1}+)|^2 +  \frac{K_1}{2} |\theta_{1,i}-u_1(y_i)|^2 \\
&& +  \sum_{i=1}^n \frac{1}{2} K(\theta_{1,i},\theta_{2,i},u_2(y_i))|\theta_{2,i}-u_2(y_i)| \\
&\leq&  \frac{2K_1+K_2}{2}  \sum_{i=1}^n v_2(u_1,(s_{i-1},s_i)) \\ && +\frac{v_1(u_2)}{2} \sup\{K(x_1,x_2,\tilde x_2);\\ && \quad\quad\quad(x_1,x_2,\tilde x_2) \in S' ,\; |x_2-\tilde x_2|\leq 
\max_{j=1,\ldots,n} \Osc(u, (s_{i-1},s_i)))\}.
\end{eqnarray*}
Moreover,
\begin{eqnarray*}
|(II)|&\leq&  v_1(u_2) \sup\{|\frac{\partial G}{\partial x_2}(z_1)-\frac{\partial G}{\partial x_2}(z_2)|;\\ &&\quad\quad\quad \; z_1,z_2 \in S,\; |z_1-z_2|_\infty\leq 
\max_{j=1,\ldots,n} \Osc(u, (s_{i-1},s_i)))\},
\end{eqnarray*}
where $|\cdot|_\infty$ denotes the maximum norm in $\R^2$.
By uniform continuity of $\frac{\partial G}{\partial x_2}$ on $S$ and of $K$ on $S'$, there is a $\delta>0$ such that 
$$
\max\left\{K(x_1,x_2,\tilde x_2),\;  \left|\frac{\partial G}{\partial x_2}(x_1,x_2)-\frac{\partial G}{\partial x_2}(\tilde x_1,\tilde x_2)\right|\right\} \leq \frac{\epsilon}{16 v_1(u_2)}
$$
for every $(x_1,x_2),\, (\tilde x_1,\tilde x_2)\in S$ with $|(x_1,x_2)-(\tilde x_1,\tilde x_2)|_\infty\leq \delta$. As $u_1$ and $u_2$ are regulated, there is a partition $\lambda=\{t_0,\ldots, t_m\}$ of $[0,T]$ such that
$$
\max_{j=1,\ldots,m}  \Osc(u, (t_{j-1},t_j))\leq \delta.
$$ 
Finally, by the equivalent characterization of $W^*_2$ in Lemma 4.1 of \citet{No}, there is a partition $\tilde \lambda$ of $[0,T]$ such that 
for every refinement $\{x_1,\ldots, x_k\}$ of $\tilde\lambda$
$$
\sum_{j=1}^k v_2(u_1,(x_{j-1},x_j)) \leq  \frac{\epsilon}{8K_1 +4K_2}.
$$
We may and shall again assume that $\lambda\cup \tilde\lambda\subset \kappa(\chi)$, by passing to a refinement of $\chi$, if necessary. Then, for every refinement $\tau$ of $\chi$,
\begin{equation}\label{eq:remainder}
| R(\tau)| \leq  \frac{\epsilon}{4}.
\end{equation}
Gathering \eqref{eq:YS_decomposition}--\eqref{eq:remainder}, we observe that  $\int_0^T \frac{\partial G}{\partial x_1}(u(s))\, du_1(s)$ exists as Young-Stieltjes integral and satisfies 
the asserted chain rule. By Theorem F.2 in Part I of \citet{DN}, this integral exists as Henstock-Kurzweil integral and coincides with the Young-Stieltjes integral, provided 
the integrand belongs to $W_2([0,T])$, i.e. has finite 2-variation over $[0,T]$. However, by \eqref{eq:reg_chain_rule}, there is a constant $\tilde K\geq 0$ such that
 for every  $m\in \N$ and $0=t_0<t_1<\cdots<t_{m-1}<t_m=T$, 
$$
\sum_{j=1}^m \left|  \frac{\partial G}{\partial x_1}(u(t_j))-\frac{\partial G}{\partial x_1}(u(t_{j-1})) \right|^2\leq \tilde K (v_2(u_1)+v_1(u_2))<\infty,
$$
 recalling that the $p$-variation was defined in Remark \ref{rem:Wp}. 
\end{proof}


\begin{thebibliography}{99}
\bibitem[Al\`os et al.(2001)]{AMN}
{\sc Al\`os, E., Mazet, O.} and {\sc Nualart, D.} (2001).  Stochastic calculus
with respect to
 Gaussian processes.
{\it Ann. Probab.} {\bf 29}  766--801. MR1849177
\bibitem[Alpay and Kipnis(2013)]{AK} {\sc Alpay, D.} and  {\sc Kipnis, A.} (2013).
A generalized white noise space approach to stochastic integration for a class of Gaussian stationary increment processes. 
{\it Opuscula Math.} {\bf 33} 395--417. MR3046404
\bibitem[Bender(2003a)]{BeSPA}
{\sc Bender, C.} (2003).  An It\^o formula for generalized functionals
of a fractional
  Brownian motion with arbitrary Hurst parameter.
  {\it Stochastic Process. Appl.} {\bf 104} 81--106. 
MR1956473 
\bibitem[Bender(2003b)]{BeBernoulli} {\sc Bender, C.} (2003).
An S-transform approach to integration with respect to a fractional Brownian motion. 
{\it Bernoulli} {\bf 9} 955--983. 
MR2046814 
\bibitem[Bender(2014)]{Be2014} {\sc Bender, C.} (2014).
Backward SDEs driven by Gaussian processes. 
{\it Stochastic Process. Appl.} {\bf 124} 2892--2916. 
MR3217428
\bibitem[Carr et al.(2003)]{CGMY} 
{\sc Carr, P., Geman, H., Madan, D. B.} and {\sc Yor, M.} (2003).
Stochastic volatility for L\'evy processes. 
{\it Math. Finance} {\bf 13} 345--382. 
MR1995283 
\bibitem[Cheridito and Nualart(2005)]{CN}
{\sc Cheridito, P.} and  {\sc Nualart, D.} (2005).
Stochastic integral of divergence type with respect to fractional Brownian motion with Hurst parameter $H\in(0,1/2)$. 
{\it Ann. Inst. H. Poincar\'e Probab. Statist.} {\bf 41} 1049--1081. 
MR2172209 
%\bibitem[Biagini et al.(2004)]{BOS}
%Biagini, F., \O ksendal, B., Sulem, A., and Wallner, N. (2004)  An
%introduction to white noise theory and Malliavin calculus for
%fractional Brownian motion. Proc. R. Soc. Lond. A 460, 347--372.
\bibitem[D{e}bicki et al.(2014)]{DHJ}    {\sc D{e}bicki, K.,  Hashorva, E.} and {\sc Ji, L.} (2014). Tail asymptotics of supremum of certain Gaussian processes over threshold dependent random intervals. {\it Extremes} {\bf 17} 411--429. 
MR3252819
\bibitem[Decreusefond and \"Ust\"unel(1999)]{DU} {\sc Decreusefond, L.} and {\sc \"Ust\"unel, A. S.} (1999). Stochastic analysis of the fractional Brownian motion.  {\it Potential Anal.} {\bf 10} 177--214. 
MR1677455 
\bibitem[Dudley and  Norvai\v sa(1999)]{DN} 
{\sc Dudley, R. M.} and {\sc Norvai\v sa, R.} (1999).
{ \it Differentiability of Six Operators on Nonsmooth Functions and $p$-Variation.  Lecture Notes in Mathematics} {\bf 1703}.
Springer, Berlin. 
MR1705318 
\bibitem[Friz et al.(2016)]{Fal}
{\sc Friz, P. K.,   Gess, B., Gulisashvili, A.}  and  {\sc Riedel, S.} (2016).
The Jain--Monrad criterion for rough paths and applications to random Fourier series and non-Markovian H\"ormander theory. {\it Ann. Probab.} {\bf 44}  684--738. 
MR3456349
\bibitem[Ganti et al.(2009)]{Gal}
{\sc Ganti, V.,  Singh, A., Passalacqua, P.}  and {\sc Foufoula-Georgiou, E.} (2009).  
Subordinated Brownian motion model for sediment transport. {\it Phys. Rev. E} {\bf 80} 011111.
\bibitem[Hu et al.(2013)]{HJT} {\sc Hu, Y.,  Jolis, M.} and  {\sc Tindel, S.} (2013).
On Stratonovich and Skorohod stochastic calculus for Gaussian processes. 
{\it Ann. Probab.} {\bf 41} 1656--1693.  
MR3098687
%\bibitem[Huang and Cambanis(1978)]{HC} Huang,  T. and Cambanis, S. (1978)
%Stochastic and multiple Wiener integrals for Gaussian processes.
%Ann. Probab. 6  585--614. 
\bibitem[Jain and Monrad(1982)]{JM} {\sc Jain, N. C.} and {\sc Monrad, D.} (1982).
Gaussian quasimartingales.
{\it Z. Wahrsch. Verw. Gebiete} {\bf 59} 139--159.  MR0650607 
\bibitem[Janson(1997)]{Ja} {\sc Janson, S.} (1997). {\it Gaussian Hilbert spaces. Cambridge Tracts in Mathematics} {\bf 129}. Cambridge University Press. Cambridge. 
MR1474726 
\bibitem[Kozubowski et al.(2006)]{KMP} 
{\sc Kozubowski, T. J., Meerschaert, M.M.} and {\sc Podg\'orski, K.} (2006). 
 Fractional Laplace motion. {\it Adv. Appl. Probab.} {\bf 38} 451--464. 
MR2264952 
\bibitem[Kruk et al.(2007)]{KRV} {\sc Kruk, I., Russo, F.} and {\sc Tudor, C.} (2007).
Wiener integrals, Malliavin calculus and covariance measure structure.  
{\it J. Funct. Anal.} {\bf 249} 92--142. MR2338856 
\bibitem[Kubo(1983)]{Kubo} {\sc Kubo, I.} (1983).
It\^o formula for generalized Brownian functionals. In {\it Theory and Application of Random Fields (Bangalore, 1982). 
Lect. Notes Control Inf. Sci.} {\bf 49} 156--166. Springer, Berlin.  
MR0799940 
\bibitem[Kumar et al.(2011)]{KMV}
{\sc Kumar, A., Meerschaert, M. M.} and {\sc Vellaisamy, P.} (2011).
Fractional normal inverse Gaussian diffusion. 
{\it Statist. Probab. Lett.}  {\bf 81} 146--152. 
MR2740078 
\bibitem[Kumar and Vellaisamy(2012)]{KV} {\sc Kumar, A.} and {\sc Vellaisamy, P.} (2012).
Fractional normal inverse Gaussian process. 
{\it Methodol. Comput. Appl. Probab.} {\bf 14} 263--283. 
MR2912339 
\bibitem[Kuo(1996)]{Kuo} 
{\sc Kuo, H.-H.} (1996).
\textit{White Noise Distribution Theory.
Probability and Stochastics Series.} CRC Press, Boca Raton, FL. 
MR1387829 
\bibitem[Lebovits(2019)]{Leb}  {\sc Lebovits, J.} (2019).
Stochastic calculus with respect to Gaussian processes.  {\it Potential Anal.} {\bf 50} 1--42. MR3900844 
\bibitem[Lebovits and Vehel(2014)]{LV} {\sc Lebovits, J.} and  {\sc V\'ehel, J. L.} (2014).
White noise-based stochastic calculus with respect to multifractional Brownian motion. 
{\it Stochastics} {\bf 86} 87--124. MR3176508
\bibitem[Lei and Nualart(2012)]{LN} {\sc Lei, P.} and {\sc Nualart, D.} (2012).
Stochastic calculus for Gaussian processes and application to hitting times. 
{\it Commun. Stoch. Anal.} {\bf 6} 379--402. 
MR2988698 
\bibitem[Meerschaert et al.(2004)]{Mal}
{\sc Meerschaert, M.M., Kozubowski, T. J., Molz, F. J.} and 
{\sc Lu, S.} (2004).  Fractional Laplace model for hydraulic conductivity. {\it Geophysical Research Letters} {\bf 31} L08501.
\bibitem[Mocioalca and Viens(2005)]{MV} {\sc Mocioalca, O.} and  {\sc Viens, F.} (2005).
Skorohod integration and stochastic calculus beyond the fractional Brownian scale. 
{\it J. Funct. Anal.} {\bf 222} 385--434. 
MR2132395 
\bibitem[Norvai\v sa(2002)]{No}  {\sc Norvai\v sa, R.} (2002) Chain rules and $p$-variation.   {\it Studia Math.} {\bf 149} 197--238. MR1890731
\bibitem[Nualart(2006)]{Nualart} 
{\sc Nualart, D.} (2006).
\textit{The Malliavin Calculus and Related Topics. Probability and its Applications (New York).} 2nd ed., 
Springer, Berlin. 
MR1344217 
\bibitem[Nualart and Taqqu(2006)]{NT} {\sc Nualart, D.} and {\sc Taqqu, M.} (2006). Wick-It\^o formula for Gaussian processes. {\it Stoch. Anal. Appl.} {\bf 24} 599--614. 
MR2220074 
\bibitem[Nualart and Taqqu(2008)]{NT2} {\sc Nualart, D.} and {\sc Taqqu, M.} (2008). 
Wick-It\^o formula for regular processes and applications to the Black and Scholes formula.
{\it Stochastics} {\bf 80} 477--487. 
MR2456333 
\bibitem[Pipiras and Taqqu(2001)]{PT} {\sc Pipiras, V.} and {\sc Taqqu, M.} (2001). Are classes of deterministic integrands for fractional Brownian motion on an interval complete? {\it Bernoulli} {\bf 7} 873--897. 
MR1873833 
\bibitem[Protter(2005)]{Protter} {\sc Protter, P. E.} (2005). {\it Stochastic Integration and Differential Equations. Applications of Mathematics (New York)} {\bf 21}. {\it Stochastic Modelling and Applied Probability.} 2nd ed.,
Springer, Berlin,  
MR2273672 
%\bibitem[Russo and Vallois(2000)]{RV1} 
%Russo, F. and Vallois, P. (2000)
%Stochastic calculus with respect to continuous finite quadratic variation processes. 
%Stochastics Stochastics Rep. 70, 1--40. 
\bibitem[Samorodnitsky(1988)]{Sam} {\sc Samorodnitsky, G.} (1988). Continuity of Gaussian processes. {\it Ann. Probab.} {\bf 16} 1019--1033. 
MR0942752 
\bibitem[Sato(1999)]{Sa} {\sc Sato, K.-I.} (1999) \textit{L\'evy Processes and Infinitely Divisible Distributions.  Cambridge Studies in Advanced Mathematics} {\bf 68}. Cambridge University Press, Cambridge. Translated from the 1990 Japanese original.  MR1739520
\bibitem[Wyloma\'nska et al.(2016)]{Wal} {\sc Wyloma\'nska, A., Kumar, A., Polocza\'nski, R.} and  {\sc Vellaisamy, P.} (2016). Inverse Gaussian and its inverse process as the subordinators of fractional Brownian motion. 
{\it Phys. Rev. E} {\bf 94} 042128.
\bibitem[Yosida(1995)]{Yosida} 
{\sc Yosida, K.} (1995).
\textit{Functional Analysis. Classics in Mathematics.} Reprint of the sixth (1980) edition. Springer, Berlin. 
MR1336382 
\end{thebibliography}
\end{document}